\documentclass[A4]{amsart}

\usepackage[square,sort,comma,numbers]{natbib}
\usepackage{amsmath, amssymb, amstext, amsfonts, textcomp, amsxtra, amsbsy, amsgen, amsopn, amscd, mathrsfs, amsthm, latexsym, array}
\usepackage[textwidth=16cm,textheight=22cm,centering]{geometry}

\usepackage{color}

\usepackage[all]{xy}

\newtheorem{theorem}             {Theorem}  [section]
\newtheorem{definition} [theorem] {Definition}
\newtheorem{lemma}      [theorem]{Lemma}
\newtheorem{corollary}  [theorem]{Corollary}
\newtheorem{proposition}[theorem]{Proposition}
\newtheorem{remark} [theorem] {Remark}

\numberwithin{equation}{section} \everymath{\displaystyle}

\newcommand{\Cont}{{\rm C}}

\newcommand{\Sch}{\mathcal{S}}

\newcommand{\Cl}{\mathfrak{C}}
\newcommand{\intL}{{\rm L}}
\newcommand{\Ht}{{\rm Ht}}
\newcommand{\BesselK}{\mathcal{K}}

\newcommand{\gp}[1]{\mathbf{#1}}
\newcommand{\GL}{{\rm GL}}

\newcommand{\SO}{{\rm SO}}
\newcommand{\SU}{{\rm SU}}

\newcommand{\ag}[1]{\mathbb{#1}}
\newcommand{\Mat}{{\rm M}}

\newcommand{\Casimir}{\mathcal{C}}

\newcommand{\F}{\mathbf{F}}
\newcommand{\Place}{{\rm S}}
\newcommand{\vo}{\mathfrak{o}}
\newcommand{\vp}{\mathfrak{p}}

\newcommand{\FuncRA}{{\rm A}}
\newcommand{\ProjF}{{\rm F}}
\newcommand{\ProjG}{{\rm G}}
\newcommand{\ProjP}{{\rm P}}
\newcommand{\norm}[1][\cdot]{\lvert #1 \rvert}
\newcommand{\extnorm}[1]{\left\lvert #1 \right\rvert}
\newcommand{\Norm}[1][\cdot]{\lVert #1 \rVert}
\newcommand{\extNorm}[1]{\left\lVert #1 \right\rVert}
\newcommand{\Pairing}[2]{\langle #1, #2 \rangle}

\newcommand{\Four}[2][]{\mathfrak{F}_{#1}(#2)}

\newcommand{\rpL}{{\rm L}}
\newcommand{\rpR}{{\rm R}}
\newcommand{\Res}{{\rm Res}}
\newcommand{\Ind}{{\rm Ind}}

\newcommand{\Intw}{\mathcal{M}}
\newcommand{\IntwR}{\mathcal{R}}
\newcommand{\Eis}{{\rm E}}
\newcommand{\Pcare}{{\it P}}
\newcommand{\Cond}{\mathbf{C}}
\newcommand{\cond}{\mathfrak{c}}
\newcommand{\fin}{{\rm fin}}
\newcommand{\eisCst}{{\rm E}_{\gp{N}}}

\newcommand{\Vol}{{\rm Vol}}
\makeatletter

\newcommand{\Rmnum}[1]{\expandafter\@slowromancap\romannumeral #1@}
\makeatother

\title{A Note on Spectral Analysis for $\GL_2$: \Rmnum{1}}
\author{Han Wu}
\thanks{Research partially supported by DFG-SNF-grant 00021L\_153647}

\begin{document}
	
	\begin{abstract}
		We establish the Fourier inversion for the smooth vectors in $\intL^2(\GL_2, \omega)$ over a number field $\F$, using minimal knowledge from automorphic representation theory. We point out a possible way to establish Fourier inversion for larger classes of functions. We also point out the incompleteness of some commonly believed ``proof'' of Fourier inversion in the literature. Moreover, the explicit computation of the intertwining operator has independent interests.
	\end{abstract}
	
	\maketitle
	
	\tableofcontents

\section{Introduction}

	\subsection{Some General Remarks}

	The group representation theory gives a unifying viewpoint of many theories, such as the classical Fourier analysis on $\ag{R}$ and $\ag{R}/\ag{Z}$, the harmonic polynomials on spheres, modular forms, etc. In the case of modular forms, this new perspective gave birth to the spectral theory of automorphic representations, which provides a powerful tool for analytic number theorists among other fruitful applications.
	
	If we look carefully into the classical Fourier analysis, we may find three parts of the theory: \\
(S1) Plancherel formula: Namely, we look for a decomposition of the concerned representation into a direct integral of irreducible representations, in the sense of identifying the inner products on both sides. This is what is conventionally called the ``spectral decomposition''. \\
(S2) Fourier inversion: We ``pull back'' the abstract underlying Hilbert spaces of the irreducible representations to suitable spaces of functions on $\ag{R}$ and verify the resulting Fourier inversion formula for as large as possible subspace of functions. \\
(S3) Justification of the formulae of the spectral projections: In fact, in the part (S1) one always first uses a dense subspace to establish the Plancherel formula then extends it to the whole space of square-integrable functions. However, the formulae of the spectral projections obtained in part (S1) always make sense for a larger subspace. Hence one must check that for this larger subspace, the spectral projections obtained by $\intL^2$-extension coincide almost everywhere with the direct computation using the formulae.

	Looking into the literature of the corresponding theory for $\GL_2$, we find that the automorphic representation theorists on the analytic side were mostly interested in the part (S1) of the theory. The fundational work is implicit in the famous book \cite{JL70} and best explained in \cite[\S 4]{GJ79}. It is then largely generalized to any connected reductive group in \cite{MW95}. Although strikingly general enough, this is not sufficient for applications in analytic number theory, especially when one meets the analogous problem for the automorphic kernel function while applying the relative trace formulae. For this reason, Iwaniec carefully studied the part (S2) in the special case over $\ag{Q}$ for $\gp{K}$-invariant incomplete Eisenstein (theta) series \cite[Theorem 7.3]{Iw02} and for $\gp{K}$-invariant automorphic kernel functions \cite[Theorem 7.4]{Iw02}. In particular, the later implies the classical Kuznetsov formula \cite[Theorem 9.3]{Iw02} or more suitably \cite[Theorem 7.14]{KL13}. The Kuznetsov formula together with the Petersson formula have become a fundamental tool in analytic number theory. Countless beautiful results are based on them, among which we only mention \cite{BH08} with addendum \cite{BH14_add} for instance.
	
	Turning back to the part (S2) for $\GL_2$, there are three obvious directions of generalization: generalizing to number fields, getting rid of $\gp{K}$ (or $\gp{K}_{\infty}$)-finiteness and changing the relative trace formulae. The first is natural due to the development of adelic language. It should be noted that the relative trace formula approach is not the only way of generalizing the classical Kuznetsov and Petersson formulae (see \cite{Ma13} for the Poincar\'e series approach). The second and the third are also naturally interesting for the theory itself. Whether or not they give interesting applications, the author can not foresee without trying them seriously. This is what the author intends to do in the current paper. Actually we have already done the first and second for smooth vectors \cite[Theorem 2.16]{Wu14} and applied it to a subconvexity problem \textit{loc.cit.} following the method initiated by \cite{MV10}. We give an alternative and simpler proof in this paper, which avoids using the delicate Whittaker-Kirillov theory. Hopefully the techniques can also apply to the automorphic kernel functions.

	\subsection{Formalism of Spectral Decomposition}
	
	Let $(\rpR, V_{\rpR})$ be a unitary representation of a locally compact group $\gp{G}$.. Let $\widehat{\gp{G}}$ be the unitary dual of $\gp{G}$ which is naturally equipped with the Fell topology. The theory of spectral decomposition is to find and establish:
\begin{itemize}
	\item[(1)] A (Plancherel) measure $d\mu = d\mu_{\rpR}$ on $\widehat{\gp{G}}$;
	\item[(2)] For each $\pi \in \widehat{\gp{G}}$ in the support of $d\mu$, a $\gp{G}$-intertwiner $\ProjF_{\pi}: V_{\rpR} \to V_{\pi}$ called the \emph{spectral projectors}, where $V_{\pi}$ is the underlying Hilbert space of $\pi$ with norm $\Norm_{\pi}$;
	\item[(3)] For every $v \in V_{\rpR}$, $\ProjF_{\pi}(v)$ is well-defined for $\pi$ outside a set with $d\mu$-measure $0$, and we have the Plancherel formula
\begin{equation}
	\int_X \norm[v(x)]^2 dx = \int_{\widehat{\gp{G}}} \Norm[\ProjF_{\pi}(v)]_{\pi}^2 d\mu(\pi).
\label{Plancherel}
\end{equation}
\end{itemize}
	Usually, this is achieved by first considering a nice subspace $V \subset V_{\rpR}$ and establishing (\ref{Plancherel}) for $V$; then one can extend the Plancherel formula to $V_{\rpR}$ by the density of $V$. We also usually get (S2) for $V$.

	\subsection{Formalism of Fourier Inversion}
	
	Let $X$ be a topological space on which a locally compact group $\gp{G}$ acts from the right. Let $dx$ be a $\gp{G}$-invariant measure on $X$. The space $\intL^2(X,dx)$ of square integrable $\ag{C}$-valued functions is naturally equipped with a unitary $\gp{G}$-action. As a special case of the setting in the previous subsection, we may consider $V_{\rpR}=\intL^2(X,dx)$. Practically in many cases, both $X$ and $\gp{G}$ come up with smooth structures with which the $\gp{G}$-action is compatible. The smooth structure of $\gp{G}$ allows us to define the $\gp{G}$-subspace of smooth vectors $\rpR^{\infty}$ resp. $\pi^{\infty}$ or $V_{\rpR}^{\infty}$ resp. $V_{\pi}^{\infty}$. For example, if $\gp{G}$ is a Lie group, we have differential operators associated to the Lie algebra of $\gp{G}$ and the meaning of a smooth vector is clear; while if $\gp{G}$ is totally disconnected, a smooth vector admits an open subgroup of $\gp{G}$ as its stabilizer group. During the establishment of (\ref{Plancherel}), we always have obtained some function realization operator for every $\pi$ in the support of $d\mu$
	$$ \FuncRA_{\pi}: V_{\pi}^{\infty} \to \Cont^{\infty}(X). $$
	In general, we can prove $V_{\rpR}^{\infty} \subset \Cont^{\infty}(X)$ and $\ProjF_{\pi}(V_{\rpR}^{\infty}) \subset V_{\pi}^{\infty}$ using suitable versions of Dixmier-Malliavan theorem. Thus we can define
	$$ \ProjP_{\pi}: V_{\rpR}^{\infty} \to \Cont^{\infty}(X), v \mapsto \FuncRA_{\pi}(\ProjF_{\pi}(v)). $$
\begin{definition}
	The \emph{Fourier inversion of $\rpR$ or $\intL^2(X,dx)$ for smooth vectors} is the identification as functions on $X$
	$$ v(x) = \int_{\widehat{\gp{G}}} \ProjP_{\pi}(v)(x) d\mu(\pi), \forall v \in V_{\rpR}^{\infty} $$
	with a specified convergence. Precisely, if we have the convergence
\begin{itemize}
	\item[(1)] in Cauchy principal sense with respect to $\pi \in \widehat{\gp{G}}$ and pointwise with respect to $x \in X$, then we call it a \emph{weak pointwise Fourier inversion formula};
	\item[(1')] in $\intL^1$ with respect to $d\mu(\pi)$ and pointwise with respect to $x \in X$, then we call it a \emph{pointwise Fourier inversion formula};
	\item[(2)] in Cauchy principal sense with respect to $\pi \in \widehat{\gp{G}}$ and uniform on compact subsets with respect to $x \in X$, then we call it a \emph{weak Fourier inversion formula};
	\item[(3)] in $\intL^1$ with respect to $d\mu(\pi)$ and normal with respect to $x \in X$, then we call it a \emph{strong Fourier inversion formula}.
\end{itemize}
	We shall call the right hand side the \emph{Fourier expansion} of $v$ (at $x$).
\label{FourInvDef}
\end{definition}
\begin{remark}
	Note that $\FuncRA_{\pi}$ is a mapping between two Fr\'echet spaces. The continuity of this map is sometimes automatic and seems to be too fundamental to be avoided in general. In particular, it seems to be a heart part to guarantee a strong Fourier inversion formula.
\end{remark}

	\subsection{Statement of the Main Results}
	
	Let $\F$ be a number field with ring of adeles $\ag{A}$. Let $\omega$ be a Hecke character. We denote by $\rpR_{\omega}$ the unitary representation of $\GL_2(\ag{A})$ on $\intL^2(\GL_2,\omega)$ consisting of measurable functions with central character $\omega$, square-integrable over $\GL_2(\F) \gp{Z}(\ag{A}) \backslash \GL_2(\ag{A})$, where $\gp{Z}$ is the center group of $\GL_2$. Let $\gp{K}$ be the usual maximal compact subgroup of $\GL_2(\ag{A})$. We state our main theorem as follows, which is a careful re-statement of \cite[Theorem 2.16]{Wu14}.
\begin{theorem}
	Let $\varphi \in V_{\rpR_{\omega}}^{\infty}$ be represented by a smooth function. Then the Fourier expansion of $\varphi$
	$$ \sum_{\pi \text{ cusp}} \sum_{\vec{n}} \ProjP_{\pi}(\varphi)[\vec{n}](g) + \sum_{\chi^2=\omega} \ProjP_{\chi \circ \det}(\varphi)(g) + \sum_{\xi} \int_0^{\infty} \sum_{\vec{n}} \ProjP_{iy,\xi,\omega\xi^{-1}}(\varphi)[\vec{n}](g) \frac{dy}{2\pi} $$
converges normally to $\varphi(g)$. Here ``$[\vec{n}]$'' is a natural parametrization of $\gp{K}$-isotypic types which will be introduced in Definition \ref{KProjParam}.
\label{MainThm}
\end{theorem}
\begin{remark}
	Because of the lack of the part (S3) for smooth vectors for $\GL_2$, the notation $\Pairing{\varphi}{E(s,\Phi)}$ in \cite[Theorem 2.16]{Wu14} was actually not justified. Actually we will establish (S3) for $\Cont_c^{\infty}(\GL_2,\omega)$ in Lemma \ref{FFforCcInf}. The method extends easily to the Schwartz space $\Sch(\GL_2,\omega)$.
\end{remark}
\begin{remark}
	The lack of an explicit construction of an approximation of the Dirac measure using functions in $V$ seems to be a major obstruction for the Fourier inversion of larger class of functions, as well as (S3) for larger class.
\label{RkMain}
\end{remark}


\section{Review of Spectral Decomposition}

	\subsection{Review of Classical Fourier Analysis}
	
	The standard ways (e.g. \cite[\S \Rmnum{6}.1]{Yos80} or \cite[Chapter 9]{Ru86}) of establishing the classical Fourier analysis on $\intL^2(\ag{R})$ actually does not follow the order (S1), (S2) and (S3) in the second paragraph of \S 1.1. One either uses the Schwartz functions or defines the Fourier transform via the explicit formulae for functions in $\intL^1(\ag{R})$ and establishes Fourier inversion for fairly large class of functions prior to the Plancherel formula. However, both ways do not seem to generalize well to $\GL_2$. A non standard way, which follows more closely our scheme and which does generalize to $\GL_2$ for the part (S1), goes as follows.
	
	Take $V \subset V_{\rpR} = \intL^2(\ag{R})$ consisting of the span of functions of the form $x^n e^{-ax^2}$, where $a>0,n \in \ag{N}$. Since the Fourier transform of Gaussian functions are explicitly computable, we easily verify for $v \in V$
	$$ \hat{v}(\xi) = \int_{\ag{R}} v(x) e^{-2\pi i x \xi} dx \in V \quad \& \quad \int_{\ag{R}} \norm[v(x)]^2 dx = \int_{\ag{R}} \norm[\hat{v}(\xi)]^2 d\xi. $$
\noindent We readily define for $v \in V$
	$$ \ProjF_{\xi}(v) = \hat{v}(\xi) e_{\xi} \in V_{\xi} = \ag{C} e_{\xi} \quad \& \quad \pi_{\xi}(x).e_{\xi} = e^{2\pi i x \xi} e_{\xi}. $$
Note that we readily verify (S2) for $V$ with strong Fourier inversion by direct computation if we \emph{define}
	$$ \FuncRA_{\xi}(e_{\xi})(x) = e^{2\pi i \xi x}. $$
By the (locally compact version of) Stone-Weierstrass theorem, for any $f \in \Cont_c(\ag{R})$ and any $\epsilon > 0$, one can find $v \in V$ such that $\sup_{x \in \ag{R}} \norm[f(x)e^{x^2}-v(x)] < \epsilon$. Thus
	$$ \int_{\ag{R}} \norm[f(x) - v(x)e^{-x^2}]^2 dx < \epsilon^2 \int_{\ag{R}} e^{-2x^2} dx, $$
from which we see that any function in $\Cont_c(\ag{R})$ can be approximated by functions in $V$ in the sense of $\intL^2$. Hence $V$ is dense in $V_{\rpR}$. We obtain (S1) and an extension of $\ProjF_{\xi}$ to $V_{\rpR}$. Let's write $\ProjF_{\xi}(f)=\tilde{f}(\xi)e_{\xi}$ for $f \in V_{\rpR}$. We have $\tilde{v}=\hat{v}$ for $v \in V$.

	For (S2), let $f \in \intL^2(\ag{R})$, continuous at $x=y$ and $\tilde{f} \in \intL^1(\ag{R})$. Take $v_{\epsilon}(x)= \epsilon^{-1} e^{-\pi x^2 / \epsilon^2} \in V$. We have $\tilde{v}_{\epsilon}(\xi)=\hat{v}_{\epsilon}(\xi) = e^{-\pi \epsilon^2 \xi^2}$. As $\epsilon \to 0^+$, $v_{\epsilon}$ is an approximation of the Dirac measure. We get by the dominated convergence theorem the pointwise Fourier inversion formula
	$$ f(y) = \lim_{\epsilon \to 0^+} \int_{\ag{R}} f(y+x) \overline{v_{\epsilon}(x)} dx = \lim_{\epsilon \to 0^+} \int_{\ag{R}} \tilde{f}(\xi) e^{2\pi i \xi y} e^{-\pi \epsilon^2 \xi^2} d\xi = \int_{\ag{R}} \tilde{f}(\xi) e^{2\pi i \xi y} d\xi. $$
It is easy to get stronger Fourier inversion formulae for smaller classes of functions by refining the above argument.

	Now take any $f \in \intL^1(\ag{R}) \cap \intL^2(\ag{R})$. The formula for $\hat{f}$ still makes sense. We need to check $\tilde{f}(\xi)=\hat{f}(\xi)$ a.e. $\xi$. To this end, we need
\begin{lemma}
	Let $1 \leq p < \infty$ and $f \in \intL^p(\ag{R})$, then
	$$ \lim_{\epsilon \to 0^+} \Norm[f*v_{\epsilon}-f]_p = 0. $$
\end{lemma}
\begin{proof}
	The proof goes exactly the same as that of \cite[Theorem 9.10]{Ru86}. In fact, it is a property of any approximation of the Dirac measure.
\end{proof}
\noindent Now first by the dominated convergence theorem then by the Plancherel formula, we get
	$$ \hat{f}(\xi) = \lim_{\epsilon \to 0^+} \int_{\ag{R}} f(x) e^{-2\pi i x \xi} \overline{\hat{v}_{\epsilon}(x)} dx = \lim_{\epsilon \to 0^+} \int_{\ag{R}} \tilde{f}(y) v_{\epsilon}(\xi -y) dy, \forall \xi \in \ag{R}. $$
On the other hand, the lemma implies $\int_{\ag{R}} \tilde{f}(y) v_{\epsilon}(\xi -y) dy \to \tilde{f}(\xi)$ in $\intL^2(\ag{R})$, hence for a.e. $\xi$ after taking a sub-sequence of $\epsilon$. (S3) is verified for $\intL^1(\ag{R}) \cap \intL^2(\ag{R})$.
\begin{remark}
	In the last equation, we must check but have omitted the proof of the fact
	$$ \int_{\ag{R}} v(x) e^{-2\pi i x \xi} \overline{\hat{v}_{\epsilon}(x)} dx = \int_{\ag{R}} \hat{v}(y) v_{\epsilon}(\xi -y) dy, \forall v \in V, $$
	which follows from a direct computation.
\end{remark}

	It seems to us that to establish (S2) and (S3) for classes of functions as large as the above ones, the use of an approximate of the Dirac mass $v_{\epsilon}$ could not be avoided. However if we are only interested in smaller classes of functions, we have the following alternative approach.
	
	First we verify (S3) for $h \in \Cont_c^{\infty}(\ag{R})$. Since (S2) is already verified for $v \in V$ with strong Fourier inversion, we have with normal convergence
	$$ v(x) = \int_{\ag{R}} \hat{v}(\xi) e^{2\pi i \xi x} dx. $$
Now that $f$ has compact support and is bounded, we can apply Fubini theorem and get
	$$ \int_{\ag{R}} h(x) \overline{v(x)} dx = \int_{\ag{R}} \int_{\ag{R}} h(x) \overline{\hat{v}(\xi)} e^{-2\pi i \xi x} d\xi dx = \int_{\ag{R}} \hat{h}(\xi) \overline{\hat{v}(\xi)} d\xi. $$
Furthermore, by integration by parts we see for any $n \in \ag{N}$
\begin{equation}
	\hat{h}(\xi) = (2\pi i \xi)^{-n} \int_{\ag{R}} h^{(n)}(x) e^{-2\pi i \xi x} dx \ll \norm[\xi]^{-n}.
\label{FourDecCla}
\end{equation}
Hence $\hat{h}$ is rapidly decreasing and bounded in $\xi$. In particular, it is in $\intL^2$. On the other hand, the Plancherel formula gives
	$$  \int_{\ag{R}} h(x) \overline{v(x)} dx = \int_{\ag{R}} \tilde{f}(\xi) \overline{\hat{v}(\xi)} d\xi. $$
Thus the $\intL^2$ function $\tilde{h} - \hat{h}$ is orthogonal to $\hat{v}$ for any $v \in V$. But the density of $V$ implies that of $\hat{V}$. Therefore $\tilde{h} - \hat{h} = 0$ in the sense of $\intL^2$, hence almost everywhere.

	Then we verify (S2) for $f \in V_{\rpR}^{\infty}=W^{\infty,2}(\ag{R})$. Since $\frac{d^n}{dx^n}f \in V_{\rpR}$ and
\begin{equation}
	\ProjF_{\xi}(\frac{d^n}{dx^n}f) = \frac{d^n}{dx^n} \ProjF_{\xi}(f) = (2\pi i \xi)^n \tilde{f}(\xi) e_{\xi} \Rightarrow (2\pi i \xi)^n \tilde{f}(\xi) \in \intL^2(\ag{R})
\label{PrDecCla}
\end{equation}
by the proof of \cite[Proposition 1.4]{CP90} (which is a consequence of the Dixmier-Malliavan theorem), we deduce that the Fourier expansion of $f$ satisfies
	$$ \int_{\ag{R}} \extnorm{\tilde{f}(\xi) e^{2\pi i \xi x}} d\xi \leq \left( \int_{\ag{R}} (1+(2\pi \xi)^2)^2 \norm[\tilde{f}(\xi)]^2 d\xi \right)^{\frac{1}{2}} \left( \int_{\ag{R}} (1+(2\pi \xi)^2)^{-2} d\xi \right)^{\frac{1}{2}} < \infty, $$
thus converges (at least) normally in $x$. In particular, it defines a continuous function in $x$, as well as the associated absolute integral. For any $h \in \Cont_c^{\infty}(\ag{R})$ we can apply Fubini and Plancherel to get
	$$ \int_{\ag{R}} \overline{h(x)} \int_{\ag{R}} \tilde{f}(\xi) e^{2\pi i \xi x} d\xi dx =  \int_{\ag{R}} \tilde{f}(\xi) \overline{\hat{h}(\xi)} d\xi = \int_{\ag{R}} \tilde{f}(\xi) \overline{\tilde{h}(\xi)} d\xi = \int_{\ag{R}} f(x) \overline{h(x)} dx, $$
where we have used (S3) for $h$. Hence both $f$ and its Fourier expansion define the same functional on $\Cont_c^{\infty}(\ag{R})$ and they are equal in the sense of distributions. They must be equal in the sense of $\intL_{{\rm loc}}^1(\ag{R})$, hence equal everywhere since they both are continuous functions.

	The above way is somewhat ``disguised''. But it seems to be convenient for generalization to $\GL_2$.
\begin{remark}
	The equations (\ref{FourDecCla}) and (\ref{PrDecCla}) look similar, especially when one takes into account that $\Cont_c^{\infty}(\ag{R}) \subset V_{\rpR}^{\infty}$. But they do have different nature. Precisely, (\ref{FourDecCla}) is calculus as we proved it by integration by parts. (\ref{PrDecCla}) is representation theoretic. A more elementary proof without appealing to Dixmier-Malliavan goes as follows. By definition we have in the sense of $\intL^2$ in $x$
	$$ \lim_{t \to 0} \frac{1}{t} \left\{ \rpR(t)f -f \right\} = \frac{d}{dx}f. $$
Taking Fourier transform, this amounts to the equation in the sense of $\intL^2$ in $\xi$
	$$ \lim_{t \to 0} \frac{1}{t} \left\{ \tilde{f}(\xi)e^{2\pi i \xi t} -\tilde{f}(\xi) \right\} = \widetilde{\frac{d}{dx}f}(\xi). $$
This implies the equality almost everywhere in $\xi$, i.e. (\ref{PrDecCla}) for $n=1$.
\end{remark}

	\subsection{Review of Spectral Decomposition for $\GL_2$}
	
	We summarize and re-interpret the concerned result in \cite[Section 3 \& 4]{GJ79} as follows.
\begin{remark}
	Throughout this paper, we fix a section $s_{\F}: \ag{R}_+ \to \F^{\times} \backslash \ag{A}^{\times}, t \mapsto t^+$ of the adelic norm map $\F^{\times} \backslash \ag{A}^{\times} \to \ag{R}_+, y \mapsto \norm[y]_{\ag{A}}$. Upon twisting by a suitable $\norm_{\ag{A}}^{i\alpha}, \alpha \in \ag{R}$, we can only consider those central characters $\omega$ trivial on the image of $s_{\F}$. We also normalize $\xi$'s that will appear later so that they are also trivial on the image of $s_{\F}$.
\label{CharNorm}
\end{remark}
\begin{definition}
	Let $V \subset V_{\rpR_{\omega}}$ be the direct sum of $\Cl_0$ (subspace of cuspidal forms) consisting of functions
	$$ \varphi \in \Cont_c^{\infty}(\GL_2, \omega), \int_{\F \backslash \ag{A}} \varphi(n(x)g) dx = 0, \forall g \in \GL_2(\ag{A}); $$
	and $\Cl_c$ consisting of incomplete theta series
\begin{equation}
	\Pcare(f)(g) = \sum_{\gamma \in \gp{B}(\F) \backslash \GL_2(\F)} f(\gamma g), f \in \Cont_c^{\infty}(\gp{N}(\ag{A})\gp{B}(\F) \backslash \GL_2(\ag{A}),\omega).
\label{PSDef}
\end{equation}
\label{FInvSpGL2}
\end{definition}
\noindent Note that $V \subset \Cont_c^{\infty}(\GL_2, \omega)$, and is dense in $V_{\rpR_{\omega}}=\intL^2(\GL_2, \omega)$. Since the delicate part is the continuous spectrum, we concerntrate ourselves in it. For any character $\xi$ of $\F^{\times} \backslash \ag{A}^{(1)}$ (extended by triviality on the image of $s_{\F}$) and $P(f) \in \Cl_c$, define
	$$ \hat{f}(s,\xi,\omega\xi^{-1})(g) = \int_{\ag{R}_+ \times \F^{\times} \backslash \ag{A}^{(1)}} f(a(t^+y)g) t^{-s-\frac{1}{2}} \xi^{-1}(y) \frac{dt}{t} d^{\times}y. $$
The integral converges for all $s \in \ag{C}$ since $f$ is of compact support. We also have the Mellin inversion
	$$ f(g) = \sum_{\xi} \int_{\Re s = c} \hat{f}(s,\xi,\omega\xi^{-1})(g) \frac{ds}{2\pi i}, $$
the sum over $\xi$ being finite due to the $\gp{K}_{\fin}$-finiteness of $f$. For $c \gg 1$ we can change the order of summation and get with normal convergence
	$$ P(f)(g) = \sum_{\xi} \int_{\Re s = c} \Eis(\hat{f}(s,\xi,\omega\xi^{-1}))(g) \frac{ds}{2\pi i}, \text{ with} $$
	$$ \Eis(\hat{f}(s,\xi,\omega\xi^{-1}))(g) = \sum_{\gamma \in \gp{B}(\F) \backslash \GL_2(\F)} \hat{f}(s,\xi,\omega\xi^{-1})(\gamma g). $$
The analyticity of $s \mapsto \Eis(\hat{f}(s,\xi,\omega\xi^{-1}))(g)$ is governed by its constant term
	$$ \eisCst(\hat{f}(s,\xi,\omega\xi^{-1}))(g) = \hat{f}(s,\xi,\omega\xi^{-1})(g) + \Intw \hat{f}(s,\xi,\omega\xi^{-1})(g), $$
where $\Intw = \Intw(s,\xi,\omega\xi^{-1}): V_{s,\xi,\omega\xi^{-1}} \to V_{-s,\omega\xi^{-1},\xi}$ is the analytic continuation in $s$ of the operator
\begin{equation}
	\Intw \hat{f}(s,\xi,\omega\xi^{-1})(g) = \int_{\ag{A}} \hat{f}(s,\xi,\omega\xi^{-1})(wn(x)g) dx,
\label{IntwDef}
\end{equation}
which intertwines the $\GL_2(\ag{A})$-actions of $\pi_{s,\xi,\omega\xi^{-1}}=\pi(\xi \norm_{\ag{A}}^s, \omega\xi^{-1} \norm_{\ag{A}}^{-s})$ and $\pi_{-s,\omega\xi^{-1},\xi}$. It has the following properties:
\begin{itemize}
	\item[(1)] It is meromorphic for $\Re s > -1/2$ and admits a possible simple pole at $s=1/2$ only if $\omega=\xi^2$ (c.f. Remark \ref{CharNorm}), in which case the residue is
\begin{align*}
	\Res_{s=\frac{1}{2}} \Intw \hat{f}(s,\xi,\xi)(g) &= -\frac{\Lambda_{\F}^*(0)}{2\Lambda_{\F}(2)} \cdot \int_{\gp{K}} \hat{f}(\frac{1}{2},\xi,\xi)(\kappa) \xi(\det \kappa)^{-1} d\kappa \cdot \xi(\det g) \\
	&= c_{\F} \cdot \int_{\GL_2(\F)\gp{Z}(\ag{A}) \backslash \GL_2(\ag{A})} P(f)(x) \overline{\xi(\det x)} dx \cdot \xi(\det g).
\end{align*}
	(Hence $c_{\F}^{-1}=\Vol(\GL_2(\F)\gp{Z}(\ag{A}) \backslash \GL_2(\ag{A}))$, as remarked in the line just before \cite[\S 5]{GJ79}.)
	\item[(2)] On the vertical line $s=iy \in i\ag{R}$, it is a unitary $\GL_2(\ag{A})$-intertwiner and we have (\cite[(8.30)]{Ge75})
	$$ \Eis(\Intw \hat{f}(iy,\xi,\omega\xi^{-1})) = \Eis(\hat{f}(iy,\xi,\omega\xi^{-1})). $$
	\item[(3)] $\Intw \hat{f}(\sigma+i\tau,\xi,\omega\xi^{-1})(g)$ is rapidly decreasing with respect to $\norm[\tau] \to \infty$ for $\sigma$ lying in any fixed compact sub-interval of $(-\epsilon, +\infty)$ and uniformly in $g$.
\end{itemize}
We state and postpone the proof of the following proposition, which is a generalization of \cite[Theorem 7.3]{Iw02}.
\begin{proposition}
	For any $P(f) \in \Cl_c$, we have the strong Fourier inversion
	$$ P(f)(g) = \sum_{\xi} \int_{-\infty}^{\infty} \Eis(\hat{f}(iy,\xi,\omega\xi^{-1}))(g) \frac{dy}{2\pi} +  c_{\F} \sum_{\chi: \chi^2=\omega} \Pairing{P(f)}{\chi \circ \det} \cdot \chi \circ \det(g). $$
\label{FourInvIncompTheta}
\end{proposition}
\noindent We tentatively define
\begin{itemize}
	\item[(1)] $\ProjF_{iy,\xi,\omega\xi^{-1}} = \ProjF_{\pi_{iy,\xi,\omega\xi^{-1}}}: \Cl_c \to V_{iy,\xi,\omega\xi^{-1}}$ by
	$$ \ProjF_{iy,\xi,\omega\xi^{-1}}(\Pcare(f)) = \hat{f}(iy,\xi,\omega\xi^{-1}) + \Intw \hat{f}(-iy,\omega\xi^{-1},\xi), $$
	and verify readily that $\ProjF_{iy,\xi,\omega\xi^{-1}}(\Cl_c) \subset V_{iy,\xi,\omega\xi^{-1}}^{\infty}$.
	\item[(2)] $\FuncRA_{iy,\xi,\omega\xi^{-1}}: V_{iy,\xi,\omega\xi^{-1}}^{\infty} \to \Cont^{\infty}(\GL_2,\omega)$ by the analytic continuation of $\FuncRA_{s,\xi,\omega\xi^{-1}}: V_{s,\xi,\omega\xi^{-1}}^{\infty} \to \Cont^{\infty}(\GL_2,\omega)$ via flat sections $f_s$ with $f \in V_{0,\xi,\omega\xi^{-1}}^{\infty}$
	$$ \FuncRA_{s,\xi,\omega\xi^{-1}}(f_s)(g) = \Eis(f_s)(g) = \Eis(s,f)(g). $$
	\item[(3)] for $\chi^2=\omega$, $\ProjF_{\chi \circ \det}: \Cl_c \to V_{\chi}=\ag{C}e_{\chi}$ with $\pi_{\chi}(g).e_{\chi} = \chi(\det g) e_{\chi}$ by
	$$ \ProjF_{\chi \circ \det}(P(f)) = c_{\F}^{\frac{1}{2}} \Pairing{P(f)}{\chi \circ \det} e_{\chi}. $$
	\item[(4)] for $\chi^2=\omega$, $\FuncRA_{\chi \circ \det}: V_{\chi}^{\infty}=V_{\chi} \to \Cont^{\infty}(\GL_2,\omega)$ by
	$$ \FuncRA_{\chi \circ \det}(e_{\chi})(g) = c_{\F}^{\frac{1}{2}} \chi \circ \det(g). $$
\end{itemize}
Therefore Proposition \ref{FourInvIncompTheta} becomes a strong Fourier inversion formula for $\Cl_c$.
\begin{remark}
	By Rankin-Selberg unfolding, one can also verify for any $e \in V_{0,\xi,\omega\xi^{-1}}^{\infty}$
	$$ \Pairing{\Pcare(f)}{\Eis(iy,\xi,\omega\xi^{-1};e)} = \Pairing{\ProjF_{iy,\xi,\omega\xi^{-1}}(\Pcare(f))}{e_{iy}}_{iy,\xi,\omega\xi^{-1}}, $$
	with Eisenstein series for flat sections defined in Definition \ref{FlatEisDef}.
\label{IncompThetaAdj}
\end{remark}

	On the other hand, we compute the inner product of two incomplete theta series using the Rankin-Selberg unfolding technique:
\begin{align*}
	\Pairing{P(f_1)}{P(f_2)} &= \int_{\gp{N}(\F) \gp{Z}(\ag{A}) \backslash \GL_2(\ag{A})} P(f_1)_{\gp{N}}(g) \overline{f_2(g)} dg \\
	&= \int_{\gp{N}(\F) \gp{Z}(\ag{A}) \backslash \GL_2(\ag{A})} f_1(g) \overline{f_2(g)} dg + \int_{\gp{N}(\F) \gp{Z}(\ag{A}) \backslash \GL_2(\ag{A})} \Intw f_1(g) \overline{f_2(g)} dg,
\end{align*}
	$$ \text{with} \quad \Intw f_1(g) = \int_{\ag{A}} f_1(wn(x)g) dx = \sum_{\xi} \int_{\Re s = c} \Intw f_1(s,\xi,\omega\xi^{-1})(g) \frac{ds}{2\pi i}, \text{for } c \gg 1. $$
For the first summand, we apply Plancherel for Mellin transform on the unitary axis; for the second, we insert the Mellin inversion for $f_1$, interchange the order of integrals and shift the vertical integral on $s$ to the unitary axis. Thus we obtain
\begin{align*}
	\Pairing{P(f_1)}{P(f_2)} &= \sum_{\xi} \int_0^{\infty} \Pairing{\ProjF_{iy,\xi,\omega\xi^{-1}}(P(f_1))}{\ProjF_{iy,\xi,\omega\xi^{-1}}(P(f_2))}_{iy} \frac{dy}{2\pi} + \\
	&\quad \sum_{\chi: \chi^2=\omega} \Pairing{\ProjF_{\chi \circ \det}(P(f_1))}{\ProjF_{\chi \circ \det}(P(f_2))}_{\chi}.
\end{align*}
The right hand side is first interpreted in the Cauchy principal sense, then the usual sense of $\intL^1$ since we know that $\Intw(iy,\xi,\omega\xi^{-1})$ is unitary. Hence we established the Plancherel formula for $V$ and can extend its validity to $\intL^2(\GL_2,\omega)$ by density, i.e. (S1) for $\GL_2$.
\begin{remark}
	In the last contour shift in $s$, one should have studied the growth in $\norm[t]$ of $\Intw(\sigma+it,\xi,\omega\xi^{-1})$ for $0 \leq \sigma < 2$. In fact, Proposition \ref{FourInvIncompTheta} will follow from this study of growth as well as the growth of the derivative $\Intw'(it,\xi,\omega\xi^{-1})$. It should be possible to study both growth without explicitly computing them. But since the explicit computation will have other applications, we will just make it.
\end{remark}

\section{Fourier Inversion for Smooth Vectors}

	\subsection{Preliminaries}
	
	A striking difference between the theory for $\GL_2$ and the classical Fourier analysis is Proposition \ref{FourInvIncompTheta}, whose counterpart in the classical case is trivial. One obvious proof consists of shifting the integral from $\Re s=c \gg 1$ to $\Re s = 0$. To this end, we must bound $\Eis(\hat{f}(s,\xi,\omega\xi^{-1}))(g)$ in this region. One possible way consists of bounding the Eisenstein series induced from unitary flat sections. With these bounds, it turns out that we can also establish the normal convergence of the Fourier expansion of a smooth vector, and the most delicate part is the bound on the axis $\Re s = 0$.
	
	We give a family of projectors $\ProjP[\vec{n}] = \otimes_v' \ProjP[n_v]$ of $\gp{K}$-representations parametrized by $\vec{n}=(n_v)_v \in \oplus_{v \in \Place(\F)} \ag{Z}$ as follows:
\begin{itemize}
	\item[(1)] At $v \mid \infty$ a complex place: $n_v \geq 0$ and $\ProjP[n_v]$ is the projector onto the unitary irreducible representation of $\SU_2(\ag{C})$ on the space $\ag{C}[x,y]_{n_v}$ of homogeneous polynomials of degree $n_v$.
	\item[(2)] At $v \mid \infty$ a real place: $n_v \in \ag{Z}$ and $\ProjP[n_v]$ is the projector onto the character of $\SO_2(\ag{R})$ given by $\begin{pmatrix} \cos \alpha & \sin \alpha \\ -\sin \alpha & \cos \alpha \end{pmatrix} \mapsto e^{in_v \alpha}$.
	\item[(3)] At $v < \infty$: $n_v \geq 0$ and $\ProjP[n_v]$ is the projector onto the orthogonal complement of the space of $\gp{K}_v[n_v-1]$-invariant vectors in the space of $\gp{K}_v[n_v]$-invariant vectors.
\end{itemize}
	Note that $\ProjP[\vec{n}]$ can be realized as integrals on $\gp{K}$.
\begin{definition}
	(Notations) For any $\vec{N} \in \oplus_{v \in \Place(\F)} \ag{N}$, $\vec{n} \leq \vec{N}$ means $\norm[n_v] \leq N_v$ for all $v$. For any representation $\rpR$ of $\gp{K}$ with underlying Hilbert space $V_{\rpR}$, $\rpR[\vec{n}]$ means the sub-representation of $\gp{K}$ on the space $\ProjP[\vec{n}] V_{\rpR}$, and $v[\vec{n}]=\ProjP[\vec{n}]v$ for any $v \in V_{\rpR}$. We write $\rpR[\leq \vec{N}] = \oplus_{\vec{n} \leq \vec{N}} \rpR[\vec{n}]$.
\label{KProjParam}
\end{definition}
\noindent It is a fact that for any unitary irreducible $\pi$ of $\GL_2(\ag{A})$, $\pi[\vec{n}]$ is either $0$ or $\gp{K}$-irreducible, hence finite dimensional.

	We shall study the the continuity of $\FuncRA_{s,\xi,\omega\xi^{-1}}: V_{s,\xi,\omega\xi^{-1}}^{\infty} \to \Cont^{\infty}(\GL_2,\omega)$ for $s \in \ag{C}$. More precisely, we shall identify $V_{s,\xi,\omega\xi^{-1}}$ with $V_{0,\xi,\omega\xi^{-1}}$ via flat sections and estimate the smooth Eisenstein series in terms of $s$ and the $\gp{K}$-structure on $V_{0,\xi,\omega\xi^{-1}}$. In particular, we shall write $V_{s,\xi,\omega\xi^{-1}}^{\infty}$ resp. $V_{s,\xi,\omega\xi^{-1}}^{\fin}$ as flat sections coming from functions in $V_{0,\xi,\omega\xi^{-1}}^{\infty}$ resp. $V_{0,\xi,\omega\xi^{-1}}^{\fin}$. So for $f \in V_{0,\xi,\omega\xi^{-1}}^{\infty}$, we write the flat section as $f_s \in V_{s,\xi,\omega\xi^{-1}}^{\infty}$ and define
\begin{equation}
	\Eis(s,\xi,\omega\xi^{-1};f) = \Eis(f_s).
\label{FlatEisDef}
\end{equation}
	
\begin{definition}
	For any character $\chi$ of $\F_v^{\times}$ resp. of $\F^{\times} \backslash \ag{A}^{\times}$ trivial on $\F^{\times} \backslash \ag{A}^{(1)}$, we write $\mu=\mu(\chi) \in \ag{R}$ such that
	$$ \chi(r) = \norm[r]_v^{i\mu} \text{ for } r>0, v \mid \infty; \chi(\varpi_v) = \norm[\varpi_v]_v^{i\mu} \text{ for } v < \infty; \text{ resp. } \chi(t) = \norm[t]_{\ag{A}}^{i\mu}. $$
\label{DefMu}
\end{definition}

	For our purpose, the adelic Maass-Selberg relation is very important. Recall the truncation operator defined in \cite[(5.7)]{GJ79} and the corresponding Maass-Selberg formula \cite[(5.13)]{GJ79}, which specializing to the case $s=s_1=\bar{s}_2$ gives
\begin{align}
	\Norm[\Lambda^c \Eis(s,\xi,\omega\xi^{-1};f)]^2 &= \frac{1}{2\Re s}\left\{ \Norm[f]^2 c^{2\Re s} - \Norm[\Intw(s,\xi,\omega\xi^{-1})f]^2 c^{-2\Re s} \right\} \label{MSRel} \\
	&\quad + 1_{\omega^{-1}\xi^2(\ag{A}^{(1)})=1}\frac{2\Im \left( \Pairing{f}{\Intw(s,\xi,\omega\xi^{-1})f} c^{i(2\Im s + \mu(\omega^{-1}\xi^2))} \right)}{2\Im s + \mu(\omega^{-1}\xi^2)}, \nonumber
\end{align}
	where $f \in V_{0,\xi,\omega\xi^{-1}}^{\infty}$.
\begin{remark}
	Taking the Sobolev inequalities into account, (\ref{MSRel}) allows us to reduce the bounding of an Eisenstein series on a fixed compact subset to the estimation of the intertwining operator (and its logarithmic derivative in the case $\Re s=0$).
\end{remark}

	We shall use the (partial) Fourier transforms
\begin{equation}
	\begin{matrix}
	\Four[1]{\Phi}(x,y) = \int_{\ag{A}} \Phi(u,y) \psi(-ux) du, \quad \Four[2]{\Phi}(x,y) = \int_{\ag{A}} \Phi(x,v) \psi(-vy) dv; \\
	\Four{\Phi}=\Four[1]{\Four[2]{\Phi}}=\Four[2]{\Four[1]{\Phi}}; \\
	\widehat{\Phi}(\vec{x})=\Four[w^{-1}]{\Phi}(\vec{x}) = \Four{\Phi}(\vec{x}.w^{-1}), w=\begin{pmatrix} & -1 \\ 1 & \end{pmatrix};
	\end{matrix}
\label{FourDef}
\end{equation}
	as well as their local versions. Note that $\widehat{\widehat{\Phi}}=\Phi$. 
	
	We shall also use the following non-conventional local version of the intertwining operator. If $\Phi$ is a Schwartz function on $\ag{A}^2$, consider the following element in $V_{s,\xi,\omega\xi^{-1}}^{\infty}$
	$$ f_{\Phi}(s,\xi,\omega\xi^{-1};g) = \xi(\det g)\norm[\det g]_{\ag{A}}^{\frac{1}{2}+s} \int_{\ag{A}^{\times}} \Phi((0,t)g) \omega^{-1}\xi^2(t) \norm[t]_{\ag{A}}^{1+2s} d^{\times}t $$
	first defined for $\Re s > 0$ then meromorphically continued to $s \in \ag{C}$. We have by (\ref{IntwDef})
\begin{align*}
	\Intw f_{\Phi}(s,\xi,\omega\xi^{-1};g) &= \int_{\ag{A}} f_{\Phi}(s,\xi,\omega\xi^{-1};wn(x)g) dx \\
	&= \xi(\det g) \norm[\det g]_{\ag{A}}^{\frac{1}{2}+s} \int_{\ag{A}^{\times}} \Four[2]{R(g).\Phi}(t,0) \omega^{-1}\xi^2(t) \norm[t]_{\ag{A}}^{2s} d^{\times}t \\
	&= \xi(\det g) \norm[\det g]_{\ag{A}}^{\frac{1}{2}+s} \int_{\ag{A}^{\times}} \Four[1]{\Four[2]{R(g).\Phi}}(t,0) \omega\xi^{-2}(t) \norm[t]_{\ag{A}}^{1-2s} d^{\times}t.
\end{align*}
	Note that
\begin{align*}
	\Four[1]{\Four[2]{R(g).\Phi}}(t,0) &= \int_{\ag{A}^2} \Phi((u,v)g) \psi((u,v)w\begin{pmatrix} 0 \\ t \end{pmatrix}) du dv \\
	&= \norm[\det g]_{\ag{A}}^{-1} \int_{\ag{A}^2} \Phi(u,v) \psi((u,v)g^*w\begin{pmatrix} 0 \\ t(\det g)^{-1} \end{pmatrix}) du dv, \quad (g^*=(\det g) g^{-1}) \\
	&= \norm[\det g]_{\ag{A}}^{-1} \widehat{\Phi}((0,t(\det g)^{-1})g).
\end{align*}
	It follows that
\begin{align*}
	\Intw f_{\Phi}(s,\xi,\omega\xi^{-1};g) &= \omega\xi^{-1}(\det g)\norm[\det g]_{\ag{A}}^{\frac{1}{2}-s} \int_{\ag{A}^{\times}} \widehat{\Phi}((0,t)g) \omega\xi^{-2}(t)\norm[t]_{\ag{A}}^{1-2s} d^{\times}t \\
	&= f_{\widehat{\Phi}}(-s,\omega\xi^{-1},\xi;g).
\end{align*}
	Suggested by Tate's theory, it is convenient to introduce $\IntwR$ defined by
\begin{equation}
	\IntwR(s,\xi,\omega\xi^{-1}) = \frac{\Lambda(1+2s,\omega^{-1}\xi^2)}{\Lambda(1-2s,\omega\xi^{-2})} \Intw(s,\xi,\omega\xi^{-1}),
\label{IntwRel}
\end{equation}
	as well as its local version
\begin{align}
	\IntwR_v(s,\xi_v,\omega_v\xi_v^{-1}): &\frac{\xi_v(\det g)\norm[\det g]_v^{\frac{1}{2}+s} \int_{\F_v^{\times}} \Phi_v((0,t)g) \omega_v^{-1}\xi_v^2(t) \norm[t]_v^{1+2s} d^{\times}t}{L_v(1+2s,\omega_v^{-1}\xi_v^2)} \label{LocIntwRDef} \\
	&\mapsto \frac{\omega_v\xi_v^{-1}(\det g)\norm[\det g]_v^{\frac{1}{2}-s} \int_{\F_v^{\times}} \widehat{\Phi_v}((0,t)g) \omega_v\xi_v^{-2}(t)\norm[t]_v^{1-2s} d^{\times}t}{L_v(1-2s,\omega_v\xi_v^{-2})}. \nonumber
\end{align}
	We shall use in the sequel
\begin{equation}
	\Intw_v(s,\xi_v,\omega_v\xi_v^{-1}) = \frac{L_v(1-2s,\omega_v\xi^{-2})}{L_v(1+2s,\omega_v^{-1}\xi_v^2)} \IntwR_v(s,\xi_v,\omega_v\xi_v^{-1}).
\label{LocIntwDef}
\end{equation}
	
	Recall the modified Bessel functions of the second kind (Bessel-K functions, \cite[6.22 (7)]{Wat44})
\begin{equation}
	\BesselK_{\nu}(y) = \frac{1}{2} \int_0^{\infty} e^{-y(t+t^{-1})} t^{\nu} \frac{dt}{t} = \int_0^{\infty} e^{-y(t^2+t^{-2})} t^{2\nu} \frac{dt}{t}, y > 0, \nu \in \ag{C}.
\label{BesselK}
\end{equation}

	For simplicity of notations, we omit the subscript $v$ in the following three subsections where we compute the local intertwining operators explicitly on $\gp{K}$-isotypic vectors.

	\subsection{Local $\gp{K}$-isotypic theory: complex place}
	
		\subsubsection{Hamiltonian}
	
	Let $\ag{H}$ be the Hamiltonian over $\ag{R}$. It has a matrix realization in $\Mat_2(\ag{C})$ as
	$$ \ag{H} = \left\{ \begin{pmatrix} z_1 & z_2 \\ -\bar{z}_2 & \bar{z}_1 \end{pmatrix}: z_1,z_2 \in \ag{C} \right\}. $$
	As real smooth manifolds, we have the identification (polar decomposition $r^2=\norm[z_1]^2+\norm[z_2]^2$)
	$$ \ag{C}^2 \simeq \ag{H} \simeq \ag{R}_+ \times \ag{H}^1, (z_1,z_2) \leftrightarrow (r,\kappa), $$
	where we have written
	$$ \ag{H}^1 = \left\{ \begin{pmatrix} z_1 & z_2 \\ -\bar{z}_2 & \bar{z}_1 \end{pmatrix} \in \ag{H}: \norm[z_1]^2+\norm[z_2]^2=1 \right\} \simeq \SU_2(\ag{C}). $$
	As Haar measures on $\ag{H}$, we must have for some $c>0$,
	$ dz_1 dz_2 = c r^3 dr d\kappa $ where $d\kappa$ is the probability Haar measure on $\SU_2(\ag{C})$ and $dz_i$ are the normalized Tate measure on $\ag{C}$. To calculate $c$, we integrate against $e^{-\norm[z_1]^2-\norm[z_2]^2}$ to get
	$$ 4\pi^2=\int_{\ag{C}^2} e^{-\norm[z_1]^2-\norm[z_2]^2} dz_1 dz_2 = c \int_0^{\infty} e^{-r^2} r^3 dr = \frac{c}{2} \Rightarrow dz_1dz_2 = 8\pi^2 r^3dr d\kappa. $$
	
		\subsubsection{Spherical Harmonics}
		
	Consider the regular representation $\varrho=\rpL \times \rpR$ of $\ag{H}^1 \times \ag{H}^1$ on $\intL^2(\ag{H}^1,d\kappa)$ with $\rpL$ resp. $\rpR$ the left resp. right regular representation.	Let $\ag{C}[\ag{H}^1]$ be the space of functions on $\ag{H}^1$ expressible as the restriction of a polynomial $P \in \ag{C}[z_1,z_2,\bar{z}_1,\bar{z}_2]$. It is the subspace of $\ag{H}^1 \times \ag{H}^1$-finite vectors. Hence we have an algebraic decomposition
	$$ \ag{C}[\ag{H}^1] = \bigoplus_{n \in \ag{N}} V_n \text{ or } \varrho = \bigoplus_{n \in \ag{N}} \varrho_n $$
	where $\varrho_n = \rho_n \otimes \rho_n$ with $\rho_n$ the irreducible representation of $\SU_2(\ag{C})$ on the space $\ag{C}[X,Y]_n$ of homogeneous polynomials of degree $n$. Given an integer $n_0$, the subspace
	$$ \ag{C}[\ag{H}^1,n_0] = \left\{ P \in \ag{C}[\ag{H}^1]: P(\begin{pmatrix} e^{i\alpha} & \\ & e^{-i\alpha} \end{pmatrix} \kappa) = e^{in_0 \alpha} P(\kappa), \forall \alpha \in \ag{R}/2\pi\ag{Z}, \kappa \in \ag{H}^1 \right\} $$
	is a sub-representation $\rpR_{n_0}$ of $\rpR$, which decomposes as
	$$ \ag{C}[\ag{H}^1,n_0] = \bigoplus_{n \geq \norm[n_0], 2 \mid n-n_0} V_{n_0,n} \text{ or } \rpR_{n_0} = \bigoplus_{n \geq \norm[n_0], 2 \mid n-n_0} \rho_n. $$
	We denote its completion in $\intL^2(\ag{H}^1)$ by $\intL^2(\ag{H}^1,n_0)$. The elements of a basis of the complexified Lie algebra of $\SU_2(\ag{C})$
	$$ H=\begin{pmatrix} i & 0 \\ 0 & -i \end{pmatrix}, X_{\pm} = \pm \frac{1}{2} \begin{pmatrix} 0 & -1 \\ 1 & 0 \end{pmatrix} - \frac{i}{2} \begin{pmatrix} 0 & i \\ i & 0 \end{pmatrix} $$
	act on $\ag{C}[\ag{H}^1]$ as differential operators
	$$ \rpL(H) = -i (z_1 \partial_1 - \bar{z}_1 \bar{\partial}_1 + z_2 \partial_2 - \bar{z}_2 \bar{\partial}_2); $$
	$$ \rpR(H) = i (z_1 \partial_1 - \bar{z}_1 \bar{\partial}_1 - z_2 \partial_2 + \bar{z}_2 \bar{\partial}_2); $$
	$$ \rpR(X_+) = z_2 \partial_1 - \bar{z}_1 \bar{\partial}_2; \rpR(X_-) = z_1 \partial_2 - \bar{z}_2 \bar{\partial}_1; $$
	where $\partial_i$ resp. $\bar{\partial}_i$ is the partial differential with respect to $z_i$ resp. $\bar{z}_i$. Let the spherical harmonic $\tilde{e}_{n,k}^{n_0} \in V_{n_0,n}$ correspond to the monomial $X^{n-k}Y^k$ for $0 \leq k \leq n$. $\tilde{e}_{n,0}^{n_0}$ is determined up to scalar by
	$$ \rpL(H).\tilde{e}_{n,0}^{n_0} = -in_0 \tilde{e}_{n,0}^{n_0}, \quad \rpR(H).\tilde{e}_{n,0}^{n_0} = in \tilde{e}_{n,0}^{n_0}, \quad \rpR(X_-).\tilde{e}_{n,0}^{n_0} = 0. $$
	It is easy to verify that $ \tilde{e}_{n,0}^{n_0} = z_1^{\frac{n+n_0}{2}} \bar{z}_2^{\frac{n-n_0}{2}} $ satisfies the above equations. From the relations $\rpR(X_+).\tilde{e}_{n,k}^{n_0} = (n-k)\tilde{e}_{n,k+1}^{n_0}$, we deduce
\begin{equation}
	\tilde{e}_{n,k}^{n_0} = \binom{n}{k}^{-1} \sum_{j=0}^k (-1)^{k-j} \binom{\frac{n+n_0}{2}}{j} \binom{\frac{n-n_0}{2}}{k-j} z_1^{\frac{n+n_0}{2}-j} z_2^j \bar{z}_1^{k-j} \bar{z}_2^{\frac{n-n_0}{2}-(k-j)}.
\label{NNBasSU2}
\end{equation}
	The harmonics $\tilde{e}_{n,k}^{n_0}, 0 \leq k \leq n$ form an orthogonal basis of $V_{n_0,n}$ but not normal.

\begin{lemma}
	We have
	$$ \Norm[\tilde{e}_{n,k}^{n_0}]^2 = \int_{\ag{H}^1} \norm[\tilde{e}_{n,k}^{n_0}(\kappa)]^2 d\kappa = \frac{1}{n+1} \binom{n}{k}^{-1} \binom{n}{\frac{n-n_0}{2}}^{-1}. $$
	Consequently, an orthonormal basis of $(\rho_n,V_{n_0,n})$ is given by 
	$$ e_{n,k}^{n_0}=\sqrt{(n+1)\binom{n}{k}\binom{n}{\frac{n-n_0}{2}}}\tilde{e}_{n,k}^{n_0}. $$
\label{NBasSU2}
\end{lemma}
\begin{proof}
	Since $V_{n_0,n}$ is isomorphic to the standard representation $\rho_n$, we have
	$$ \frac{\Norm[\tilde{e}_{n,k}^{n_0}]^2}{\Norm[\tilde{e}_{n,0}^{n_0}]^2} = \frac{\Norm[X^{n-k}Y^k]_{\rho_n}^2}{\Norm[X^n]_{\rho_n}^2} = \frac{\int_{\ag{R}^2} e^{-x^2-y^2}x^{2(n-k)}y^{2k} dx dy}{\int_{\ag{R}^2} e^{-x^2-y^2}x^{2n} dx dy} = \binom{n}{k}^{-1}. $$
	We have two ways to calculate the integral
\begin{align*}
	\int_{\ag{C}^2} e^{-\norm[z_1]^2-\norm[z_2]^2} \norm[z_1]^{n+n_0} \norm[z_2]^{n-n_0} dz_1 dz_2 &= 4(2\pi)^2 \int_0^{\infty} e^{-r^2} r^{n+n_0+1} dr \int_0^{\infty} e^{-r^2} r^{n-n_0+1} dr \\
	&= (2\pi)^2 \left( \frac{n+n_0}{2} \right)! \left( \frac{n-n_0}{2} \right)!;
 \end{align*}
 	$$ \int_{\ag{C}^2} e^{-\norm[z_1]^2-\norm[z_2]^2} \norm[z_1]^{n+n_0} \norm[z_2]^{n-n_0} dz_1 dz_2 = 8\pi^2 \Norm[\tilde{e}_{n,0}^{n_0}]^2 \int_0^{\infty} e^{-r^2} r^{2n+3} dr = 4\pi^2 (n+1)! \Norm[\tilde{e}_{n,0}^{n_0}]^2. $$
 	Comparing the right hand sides, we conclude.
\end{proof}

		\subsubsection{Intertwining Operator}
		
	Let $\Sch_P(\ag{H})=\Sch_P(\ag{R}^4)$ be the subspace of the Schwartz function space $\Sch(\ag{H})=\Sch(\ag{R}^4)$ spanned by
\begin{equation}
	P_{\vec{n}}(\vec{z}) = e^{-2\pi (\norm[z_1]^2 + \norm[z_2]^2)} z_1^{n_1} z_2^{n_2} \bar{z}_1^{\bar{n}_1} \bar{z}_2^{\bar{n}_2}, \vec{n} \in \ag{N}^4.
\label{BasSchPC}
\end{equation}
	It is naturally equipped with an action of $\gp{K}=\SU_2(\ag{C})$, every element of which is $\gp{K}$-finite. The explicit formula for the Fourier transform being given by
	$$ \widehat{\Phi}(z_1,z_2) = \int_{\ag{C}^2} \Phi(u_1,u_2) e^{-2\pi i (z_1u_2-z_2u_1+\bar{z}_1 \bar{u}_2 - \bar{z}_2 \bar{u}_1)} du_1 du_2, $$
	we deduce that the Fourier transform interchanges
	$$ \partial_1 \leftrightarrow -2\pi i z_2, \partial_2 \leftrightarrow 2\pi i z_1, \bar{\partial}_1 \leftrightarrow -2\pi i \bar{z}_2, \bar{\partial}_2 \leftrightarrow -2\pi i \bar{z}_1. $$
	Consequently, the Fourier transform $\Phi \mapsto \widehat{\Phi}$ leaves $\Sch_P(\ag{H})$ stable. Moreover, for any $\kappa \in \ag{H}^1$ we have
\begin{align}
	\widehat{\kappa.\Phi}(\vec{z}) &= \int_{\ag{C}^2} \kappa.\Phi(\vec{u}) \psi_{\ag{C}}(\vec{u}w^{-1}\vec{z}^T) d\vec{u} = \int_{\ag{C}^2} \Phi(\vec{u}) \psi_{\ag{C}}(\vec{u}\kappa^{-1}w^{-1}\vec{z}^T) d\vec{u} \nonumber \\
	&= \widehat{\Phi}(\vec{z}.(w\kappa^{-1}w^{-1})^T) = \widehat{\Phi}(\vec{z}.\kappa), \label{FourKmap}
\end{align}
	since $\kappa \mapsto w (\kappa^{-1})^T w^{-1}$ is the identity map on $\ag{H}^1$. Hence the Fourier transform is a $\gp{K}$-map on both $\Sch(\ag{H})$ and $\Sch_P(\ag{H})$.
\begin{lemma}
\begin{itemize}
	\item[(1)] The family of $\gp{K}$-maps $\Sch(\ag{H}) \to V_{s,\xi,\omega\xi^{-1}}^{\infty}$ resp. $\Sch_P(\ag{H}) \to V_{s,\xi,\omega\xi^{-1}}^{\fin}$ defined by the Tate integral (first for $\Re s > 0$ then analytically continued to $s \in \ag{C}$)
	$$ \Phi \mapsto f_{\Phi}(s,\xi,\omega\xi^{-1};g) := \xi(\det g) \norm[\det g]_{\ag{C}}^{\frac{1}{2}+s} \int_{\ag{C}^{\times}} \Phi((0,t)g) \omega^{-1}\xi^2(t) \norm[t]_{\ag{C}}^{1+2s} d^{\times}t $$
	are well-defined, meromorphic in $s$ and surjective for each fixed $s,\xi,\omega$.
	\item[(2)] If we replace $f_{\Phi}$ with $\tilde{f}_{\Phi}$ defined by
	$$ \Phi \mapsto \tilde{f}_{\Phi}(s,\xi,\omega\xi^{-1};g) := \frac{f_{\Phi}(s,\xi,\omega\xi^{-1};g)}{L_{\ag{C}}(1+2s, \omega^{-1}\xi^2)}, $$
	then it is holomorphic in $s \in \ag{C}$.
\end{itemize}
\label{IntwPropC}
\end{lemma}
\begin{proof}
	(1) It is clear from the formula that these are $\gp{K}$-maps. Since the smooth resp. $\gp{K}$-finite structure of $V_{s,\xi,\omega\xi^{-1}}$ is the same as the one of $V_{0,\xi,\omega\xi^{-1}}$ via flat sections, which depends in turn only on its restriction to $\gp{K}$, and since $\Sch(\ag{H})$ is stable under the Lie algebra of $\gp{K}$, the map is well-defined. The meromorphic continuation as well as (2) are given by the theory of Tate's zeta integrals. If $\omega^{-1}\xi^2(e^{i\alpha})=e^{in_0\alpha}$ for some $n_0 \in \ag{Z}$, then it is clear that
	$$ \Res_{\gp{K}}^{\GL_2(\ag{C})} V_{s,\xi,\omega\xi^{-1}} = \intL^2(\ag{H}^1,n_0) \text{ resp. } \Res_{\gp{K}}^{\GL_2(\ag{C})} V_{s,\xi,\omega\xi^{-1}}^{\fin} = \ag{C}[\ag{H}^1,n_0]. $$
	For $\tilde{e}_{n,0}^{n_0}$ (\ref{NNBasSU2}), we take $P=(-1)^{\frac{n-n_0}{2}}P_{\vec{n}}$ (\ref{BasSchPC}) with $\vec{n}=(\frac{n-n_0}{2},0,0,\frac{n+n_0}{2})$ such that
\begin{equation}
	P(\begin{pmatrix} e^{i\alpha} & \\ & e^{-i\alpha} \end{pmatrix}\kappa) = e^{-in_0\alpha}P(\kappa), P(w^{-1}\kappa) = e^{-2\pi} \tilde{e}_{n,0}^{n_0}(\kappa), \forall \kappa \in \SU_2(\ag{C}).
\label{FinSchSecC}
\end{equation}
	The surjectivity in the $\gp{K}$-finite case follows from
\begin{align*}
	f_{P}(s,\xi,\omega\xi^{-1};\kappa) &= \int_{\ag{C}^{\times}} P((0,t)\kappa) \omega^{-1}\xi^2(t) \norm[t]_{\ag{C}}^{1+2s} d^{\times}t \\
	&= \frac{2}{\pi} \int_0^{\infty} \int_0^{2\pi} P(\begin{pmatrix} r & 0 \\ 0 & r \end{pmatrix} \begin{pmatrix} e^{i\alpha} & \\ & e^{-i\alpha} \end{pmatrix}w^{-1}\kappa) e^{in_0\alpha} \omega^{-1}\xi^2(r) r^{2+4s} \frac{dr}{r} d\alpha \\
	&= 4 \int_0^{\infty} e^{-2\pi r^2} \omega^{-1}\xi^2(r) r^{2+4s+n} \frac{dr}{r} \cdot \tilde{e}_{n,0}^{n_0}(\kappa) = \Gamma_{\ag{C}}(1+2s+i\mu(\omega^{-1}\xi^2)+\frac{n}{2}) \tilde{e}_{n,0}^{n_0}(\kappa).
\end{align*}
	For $f \in V_{0,\xi,\omega\xi^{-1}}^{\infty}$, we take $P=P_a$ for some $a>0$ defined in the polar coordinates by
	$$ P_a(r\kappa) = e^{-a(r^2+r^{-2})} f(w\kappa), r \in \ag{R}_+, \kappa \in \ag{H}^1. $$
	In other words, $P_a(h) = e^{-a(\Norm[h]^2+\Norm[h]^{-2})} f(wh/\Norm[h])$ where $h \mapsto \Norm[h]^2$ is the reduced norm in $\ag{H}$. With this expression, it is easy to verify $P_a \in \Sch(\ag{H})$. The surjectivity in the smooth case follows from
\begin{align*}
	f_{P}(s,\xi,\omega\xi^{-1};\kappa) &= \int_{\ag{C}^{\times}} P_a((0,t)\kappa) \omega^{-1}\xi^2(t) \norm[t]_{\ag{C}}^{1+2s} d^{\times}t \\
	&= 4 \int_0^{\infty} e^{-a (r^2+r^{-2})} \omega^{-1}\xi^2(r) r^{2+4s} \frac{dr}{r} \cdot f(\kappa),
\end{align*}
	and the fact that we can choose $a \in [1,2)$ such that
	$$ K_{\ag{C},a}(1+2s+i\mu(\omega^{-1}\xi^2)) := 4 \int_0^{\infty} e^{-a (r^2+r^{-2})} \omega^{-1}\xi^2(r) r^{2+4s} \frac{dr}{r} $$
	is non-vanishing.
\end{proof}
\begin{remark}
	We can write $K_{\ag{C},a}$ in terms of the Bessel K-functions (\ref{BesselK}) as
	$$ K_{\ag{C},a}(1+2s+i\mu(\omega^{-1}\xi^2)) = 4 \BesselK_{1+2s+i\mu(\omega^{-1}\xi^2)}(a). $$
\end{remark}
	If $\omega^{-1}\xi^2(e^{i\alpha})=e^{in_0\alpha}$ for some $n_0 \in \ag{Z}$, then $e_{n,k}^{n_0}$ ((\ref{NNBasSU2}) and Lemma \ref{NBasSU2}) gives rise to a flat section of $\pi_{s,\xi}$ determined by
	$$ e(s,\xi,\omega\xi^{-1};n,k;\kappa) = e_{n,k}^{n_0}(\kappa), \kappa \in \SU_2(\ag{C}). $$
\begin{corollary}
\begin{itemize}
	\item[(1)] The effect of $\IntwR(s,\xi,\omega\xi^{-1})$ (\ref{LocIntwRDef}) on the $\gp{K}$-finite flat sections is given by
	$$ \IntwR e(s,\xi,\omega\xi^{-1};n,k) = \mu(s,\xi,\omega\xi^{-1};n) e(-s,\omega\xi^{-1},\xi;n,k), \text{ with} $$
	$$ \mu(s,\xi,\omega\xi^{-1};n) = \frac{\Gamma_{\ag{C}}(1+2s+i\mu(\omega^{-1}\xi^2)+\frac{\norm[n_0]}{2})}{\Gamma_{\ag{C}}(1-2s-i\mu(\omega^{-1}\xi^2)+\frac{\norm[n_0]}{2})} \frac{\Gamma_{\ag{C}}(1-2s-i\mu(\omega^{-1}\xi^2)+\frac{n}{2})}{i^{n_0} \Gamma_{\ag{C}}(1+2s+i\mu(\omega^{-1}\xi^2)+\frac{n}{2})}, $$
	where $n_0 \in \ag{Z}$ is determined by $\omega^{-1}\xi^2(e^{i\alpha})=e^{in_0\alpha}$.
	\item[(2)] We have for $y \in \ag{R}$, $\norm[\mu(iy,\xi,\omega\xi^{-1};n)]=1$ and
	$$ \extnorm{\mu'(iy,\xi,\omega\xi^{-1};n)} \leq \left\{ \begin{matrix} 0 & \text{if } n=\norm[n_0] \\ 4 \left( \frac{2}{\norm[n_0]+2} + \log \frac{n}{\norm[n_0]+2} \right) & \text{if } n \geq \norm[n_0]+2. \end{matrix} \right. $$
\end{itemize}
\label{LocIntwCalCpx}
\end{corollary}
\begin{proof}
	$P$ associated with $\tilde{e}_{n,0}^{n_0}$ given by (\ref{FinSchSecC}) satisfies
	$$ \widehat{P}  = \frac{(-1)^{\frac{n-n_0}{2}} \bar{\partial}_1^{\frac{n+n_0}{2}} \partial_2^{\frac{n-n_0}{2}} P_{\vec{0}}}{(2\pi i)^{\frac{n-n_0}{2}} (-2\pi i)^{\frac{n+n_0}{2}}} = (-1)^{\frac{n+n_0}{2}} i^{-n_0} z_1^{\frac{n+n_0}{2}} \bar{z}_2^{\frac{n-n_0}{2}} P_{\vec{0}}, $$
	which is associated with $i^{-n_0}\tilde{e}_{n,0}^{-n_0}$ by (\ref{FinSchSecC}). Using the formulas given by the lemma, we get
\begin{align*}
	\Intw \tilde{e}(s,\xi,\omega\xi^{-1};n,0) &= \frac{\Intw f_P(s,\xi,\omega\xi^{-1})}{\Gamma_{\ag{C}}(1+2s+i\mu(\omega^{-1}\xi^2)+\frac{n}{2})} = \frac{f_{\widehat{P}}(-s,\omega\xi^{-1},\xi)}{ \Gamma_{\ag{C}}(1+2s+i\mu(\omega^{-1}\xi^2)+\frac{n}{2})} \\
	&= \frac{\Gamma_{\ag{C}}(1-2s+i\mu(\omega\xi^{-2})+\frac{n}{2})}{ i^{n_0} \Gamma_{\ag{C}}(1+2s+i\mu(\omega^{-1}\xi^2)+\frac{n}{2})} \tilde{e}(-s,\omega\xi^{-1},\xi;n,0).
\end{align*}
	Since $\Intw$ is a $\gp{K}$-map, we can generate $(n,k)$'s from $(n,0)$ with differentials of $\gp{K}$ in the same way on both sides and conclude for (1). Now writing $\mu=\mu(\omega^{-1}\xi^2)$ for simplicity, we deduce
	$$ \mu(iy,\xi,\omega\xi^{-1};n) = \frac{1}{i^{n_0}} \prod_{k=0}^{\frac{n-\norm[n_0]}{2}-1} \frac{1-i(2y+\mu)+\frac{\norm[n_0]}{2}+k}{1+i(2y+\mu)+\frac{\norm[n_0]}{2}+k}; $$
	$$ \frac{\mu'(iy,\xi,\omega\xi^{-1};n)}{\mu(iy,\xi,\omega\xi^{-1};n)} = -4 \sum_{k=0}^{\frac{n-\norm[n_0]}{2}-1} \frac{1+\frac{\norm[n_0]}{2}+k}{(2y+\mu)^2+(1+\frac{\norm[n_0]}{2}+k)^2}. $$
	The asserted bound in (2) then becomes obvious.
\end{proof}

	\subsection{Local $\gp{K}$-isotypic theory: real place}
	
	The real case is similar to and much simpler than the complex case. $\ag{C}$ as an $\ag{R}$-algebra has a matrix realization in $\Mat_2(\ag{R})$ as
	$$ \ag{C} = \left\{ \begin{pmatrix} x_1 & x_2 \\ -x_2 & x_1 \end{pmatrix}: x_1,x_2 \in \ag{R} \right\}. $$
	As real smooth manifolds, we have the identification (polar decomposition $r^2=\norm[x_1]^2+\norm[x_2]^2$)
	$$ \ag{R}^2 \simeq \ag{C} \simeq \ag{R}_+ \times \ag{C}^1, (x_1,x_2) \leftrightarrow (r,\kappa), $$
	where we have written
	$$ \ag{C}^1 = \left\{ \begin{pmatrix} x_1 & x_2 \\ -x_2 & x_1 \end{pmatrix} \in \ag{C}: \norm[x_1]^2+\norm[x_2]^2=1 \right\} \simeq \SO_2(\ag{R}). $$
	As Haar measures on $\ag{C}$, we have for $d\kappa$ the probability Haar measure on $\SO_2(\ag{R})$ the relation
	$$ dx_1dx_2 = 2\pi rdr d\kappa. $$
	Consider the regular representation $\varrho=\rpL \times \rpR$ of $\ag{C}^1 \times \ag{C}^1$ on $\intL^2(\ag{C}^1,d\kappa)$. Let $\ag{C}[\ag{C}^1]$ be the space of functions on $\ag{C}^1$ expressible as the restriction of a polynomial $P \in \ag{C}[x_1,x_2]$. It is the subspace of $\ag{C}^1 \times \ag{C}^1$-finite vectors. Hence we have an algebraic decomposition
	$$ \ag{C}[\ag{C}^1] = \bigoplus_{n \in \ag{Z}} V_n \text{ or } \varrho = \bigoplus_{n \in \ag{Z}} \varrho_n $$
	where $\varrho_n = \rho_n \otimes \rho_n$ with $\rho_n$ the irreducible representation of $\SO_2(\ag{R})$ given by the character
	$$ \begin{pmatrix} e^{i\alpha} & \\ & e^{i\alpha} \end{pmatrix} \mapsto e^{in\alpha} .$$
	Given $n_0 \in \{0,1\}$, the subspace
	$$ \ag{C}[\ag{C}^1,n_0] = \left\{ P \in \ag{C}[\ag{H}^1]: P(\begin{pmatrix} -1 & \\ & -1 \end{pmatrix} \kappa) = (-1)^{n_0} P(\kappa), \forall \kappa \in \ag{H}^1 \right\} $$
	is a sub-representation $\rpR_{n_0}$ of $\rpR$, which decomposes as
	$$ \ag{C}[\ag{C}^1,n_0] = \bigoplus_{2 \mid n-n_0} V_{n_0,n} \text{ or } \rpR_{n_0} = \bigoplus_{2 \mid n-n_0} \rho_n. $$
	We denote its completion in $\intL^2(\ag{C}^1)$ by $\intL^2(\ag{C}^1,n_0)$. $V_{n_0,n}$ is one dimensional with a normalized basis element
\begin{equation}
	e_n(z) = z^n \text{ if } n \geq 0; \bar{z}^{-n} \text{ if } n < 0.
\label{NBasSO2}
\end{equation}
	Let $\Sch_P(\ag{C})=\Sch_P(\ag{R}^2)$ be the subspace of the Schwartz function space $\Sch(\ag{C})=\Sch(\ag{R}^2)$ spanned by
\begin{equation}
	P_{n}(z) = e^{-\pi \norm[z]^2} z^{\frac{\norm[n]+n}{2}} \bar{z}^{\frac{\norm[n]-n}{2}}, n \in \ag{Z}.
\label{BasSchPR}
\end{equation}
	It is naturally equipped with an action of $\gp{K}=\SO_2(\ag{R})$, every element of which is $\gp{K}$-finite. The explicit formula for the Fourier transform being given by
	$$ \widehat{\Phi}(x_1,x_2) = \int_{\ag{R}^2} \Phi(u_1,u_2) e^{-2\pi i (x_1u_2-x_2u_1)} du_1 du_2, \text{ i.e.,} $$
	$$ \widehat{\Phi}(z) = \int_{\ag{R}^2} \Phi(u) e^{-\pi(u\bar{z}-\bar{u}z)} du, $$
	we deduce that the Fourier transform interchanges
	$$ \partial \leftrightarrow \pi \bar{z}, \bar{\partial} \leftrightarrow -\pi z. $$
	Consequently, the Fourier transform $\Phi \mapsto \widehat{\Phi}$ leaves $\Sch_P(\ag{C})$ stable. Moreover, for any $\kappa \in \ag{C}^1$ we have
\begin{align}
	\widehat{\kappa.\Phi}(\vec{x}) &= \int_{\ag{R}^2} \kappa.\Phi(\vec{u}) \psi_{\ag{R}}(\vec{u}w^{-1}\vec{x}^T) d\vec{u} = \int_{\ag{R}^2} \Phi(\vec{u}) \psi_{\ag{R}}(\vec{u}\kappa^{-1}w^{-1}\vec{x}^T) d\vec{u} \nonumber \\
	&= \widehat{\Phi}(\vec{x}.(w\kappa^{-1}w^{-1})^T) = \widehat{\Phi}(\vec{x}.\kappa), \label{FourKmapR}
\end{align}
	since $\kappa \mapsto w (\kappa^{-1})^T w^{-1}$ is the identity map on $\ag{C}^1$. Hence the Fourier transform is a $\gp{K}$-map on both $\Sch(\ag{C})$ and $\Sch_P(\ag{C})$.
\begin{lemma}
\begin{itemize}
	\item[(1)] The family of $\gp{K}$-maps $\Sch(\ag{C}) \to V_{s,\xi,\omega\xi^{-1}}^{\infty}$ resp. $\Sch_P(\ag{C}) \to V_{s,\xi,\omega\xi^{-1}}^{\fin}$ defined by the Tate integral (first for $\Re s > 0$ then analytically continued to $s \in \ag{C}$)
	$$ \Phi \mapsto f_{\Phi}(s,\xi,\omega\xi^{-1};g) := \xi(\det g) \norm[\det g]_{\ag{R}}^{\frac{1}{2}+s} \int_{\ag{R}^{\times}} \Phi((0,t)g) \omega^{-1}\xi^2(t) \norm[t]_{\ag{R}}^{1+2s} d^{\times}t $$
	are well-defined, meromorphic in $s$ and surjective for each fixed $s,\xi,\omega$.
	\item[(2)] If we replace $f_{\Phi}$ with $\tilde{f}_{\Phi}$ defined by
	$$ \Phi \mapsto \tilde{f}_{\Phi}(s,\xi,\omega\xi^{-1};g) := \frac{f_{\Phi}(s,\xi,\omega\xi^{-1};g)}{L_{\ag{R}}(1+2s, \omega^{-1}\xi^2)}, $$
	then it is holomorphic in $s \in \ag{C}$.
\end{itemize}
\label{IntwPropR}
\end{lemma}
\begin{proof}
	The proof being quite similar to that of Lemma \ref{IntwPropC}, we only record the essential steps. Assume $\omega\xi^{-1}(-1)=(-1)^{n_0}$ for some $n_0 \in \{0,1\}$. For $e_n$ (\ref{NBasSO2}), we take $P=(-i)^n P_n$ (\ref{BasSchPR}) such that
\begin{equation}
	P(\begin{pmatrix} -1 & \\ & -1 \end{pmatrix}\kappa) = (-1)^{n_0} P(\kappa), P(w^{-1}\kappa) = e^{-\pi} e_n(\kappa), \kappa \in \SO_2(\ag{R}).
\label{FinSchSecR}
\end{equation}
	The surjectivity in the $\gp{K}$-finite case follows from
\begin{align*}
	f_{P}(s,\xi,\omega\xi^{-1};\kappa) &= \int_{\ag{R}^{\times}} P((0,t)\kappa) \omega^{-1}\xi^2(t) \norm[t]_{\ag{R}}^{1+2s} d^{\times}t \\
	&= 2 \int_0^{\infty} P(\begin{pmatrix} r & 0 \\ 0 & r \end{pmatrix} w^{-1}\kappa) \omega^{-1}\xi^2(r) r^{1+2s} \frac{dr}{r} \\
	&= 2 \int_0^{\infty} e^{-\pi r^2} \omega^{-1}\xi^2(r) r^{1+2s+\norm[n]} \frac{dr}{r} \cdot e_n(\kappa) = \Gamma_{\ag{R}}(1+2s+i\mu(\omega^{-1}\xi^2)+\norm[n]) e_n(\kappa).
\end{align*}
	For $f \in V_{0,\xi,\omega\xi^{-1}}^{\infty}$, we take $P=P_a$ for some $a>0$ defined in the polar coordinates by
	$$ P_a(r\kappa) = e^{-a(r^2+r^{-2})} f(w\kappa), r \in \ag{R}_+, \kappa \in \ag{C}^1. $$
	In other words, $P_a(z) = e^{-a(\norm[z]^2+\norm[z]^{-2})} f(-iz/\norm[z])$. With this expression, it is easy to verify $P_a \in \Sch(\ag{C})$. The surjectivity in the smooth case follows from
\begin{align*}
	f_{P}(s,\xi,\omega\xi^{-1};\kappa) &= \int_{\ag{R}^{\times}} P_a((0,t)\kappa) \omega^{-1}\xi^2(t) \norm[t]_{\ag{R}}^{1+2s} d^{\times}t \\
	&= 2 \int_0^{\infty} e^{-a (r^2+r^{-2})} \omega^{-1}\xi^2(r) r^{1+2s} \frac{dr}{r} \cdot f(\kappa),
\end{align*}
	and the fact that we can choose $a \in [1,2)$ such that
	$$ K_{\ag{R},a}(1+2s+i\mu(\omega^{-1}\xi^2)) := 2 \int_0^{\infty} e^{-a (r^2+r^{-2})} \omega^{-1}\xi^2(r) r^{1+2s} \frac{dr}{r} $$
	is non-vanishing.
\end{proof}
\begin{remark}
	We can write $K_{\ag{R},a}$ in terms of the Bessel K-functions (\ref{BesselK}) as
	$$ K_{\ag{R},a}(1+2s+i\mu(\omega^{-1}\xi^2)) = 2 \BesselK_{\frac{1}{2}+s+i\frac{\mu(\omega^{-1}\xi^2)}{2}}(a). $$
\end{remark}
	$e_n$ (\ref{NBasSO2}) gives rise to a flat section of $\pi_{s,\xi}$ determined by
	$$ e(s,\xi,\omega\xi^{-1};n;\kappa) = e_n(\kappa), \kappa \in \SO_2(\ag{R}). $$
\begin{corollary}
\begin{itemize}
	\item[(1)] The effect of $\IntwR(s,\xi,\omega\xi^{-1})$ (\ref{LocIntwRDef}) on the $\gp{K}$-finite flat sections is given by
	$$ \IntwR e(s,\xi,\omega\xi^{-1};n) = \mu(s,\xi,\omega\xi^{-1};n) e(-s,\omega\xi^{-1},\xi;n), \text{ with} $$
	$$ \mu(s,\xi,\omega\xi^{-1};n) = \frac{\Gamma_{\ag{R}}(1+2s+i\mu(\omega^{-1}\xi^2)+n_0)}{\Gamma_{\ag{R}}(1-2s-i\mu(\omega^{-1}\xi^2)+n_0)} \frac{\Gamma_{\ag{R}}(1-2s-i\mu(\omega^{-1}\xi^2)+\norm[n])}{(-1)^{\frac{\norm[n]-n}{2}} \Gamma_{\ag{R}}(1+2s+i\mu(\omega^{-1}\xi^2)+\norm[n])}, $$
	where $n_0=n_0(\omega^{-1}\xi^2)=n_0(\omega) \in \{0,1\}$ is determined by $\omega(-1)=(-1)^{n_0}$.
	\item[(2)] We have for $y \in \ag{R}$, $\norm[\mu(iy,\xi,\omega\xi^{-1};n)]=1$ and
	$$ \extnorm{\mu'(iy,\xi,\omega\xi^{-1};n)} \leq \left\{ \begin{matrix} 0 & \text{if } \norm[n]=n_0 \\ 2 \left( \frac{2}{n_0+1} + \log \frac{\norm[n]-1}{n_0+1} \right) & \text{if } \norm[n] \geq n_0+2. \end{matrix} \right. $$
\end{itemize}
\label{LocIntwCalReal}
\end{corollary}

	\subsection{Local $\gp{K}$-isotypic theory: finite place}
	
		\subsubsection{Some Fourier Transforms}
		
	Let $\chi$ be a character of $\vo^{\times}$ with conductor $\cond(\chi)$. Recall we have fixed a uniformiser $\varpi$.
\begin{lemma}
\begin{itemize}
	\item[(1)] If $\cond(\chi)>0$, then the integral $ \int_{\vo^{\times}} \chi(y) \psi(-\varpi^n y) dy $ is non-vanishing only when $n=-\cond(\psi)-\cond(\chi)$.
	\item[(2)] If we denote by
	$$ G(\chi,\psi) = \int_{\vo^{\times}} \chi(y) \psi(-\varpi^{-\cond(\psi)-\cond(\chi)}y) dy, $$
	then we have $\norm[G(\chi,\psi)]=\Cond(\chi)^{-1/2}\Cond(\psi)^{-1/2}$.
	\item[(3)] We shall write
	$$ g(\chi,\psi) = G(\chi,\psi) \Cond(\chi)^{\frac{1}{2}} \Cond(\psi)^{\frac{1}{2}}. $$
	We have $\norm[g(\chi,\psi)]=1$ and
	$$ g(\chi^{-1},\psi) = \chi(-1) \overline{g(\chi,\psi)}. $$
\end{itemize}
\end{lemma}
\begin{proof}
	Up to normalization of measures, (2) is exactly \cite[Proposition 4.6]{Wu14}. Other assertions are simple computations.
\end{proof}
\begin{definition}
	For $n \in \ag{N}$, we denote by $[\chi,n]$ resp. $[1,\geq n]$ the function on $\F^{\times}$ supported in $\varpi^n \vo^{\times}$ resp. $\varpi^n \vo$ given by
	$$ \varpi^n \vo^{\times} \to \ag{C}, \varpi^n y \mapsto \chi(y) \text{ resp. } 1_{\varpi^n \vo}. $$
\end{definition}
\begin{proposition}
\begin{itemize}
	\item[(1)] If $\cond(\chi)>0$, then we have the Fourier transform
	$$ \Four{[\chi,n]} = q^{-n} G(\chi,\psi) [\chi^{-1}, -n-\cond(\psi)-\cond(\chi)]. $$
	\item[(2)] If $\cond(\chi)=0$, then $\chi=1$ and we have the Fourier transform
	$$ \Four{[1,\geq n]} = q^{-n} \Cond(\psi)^{-\frac{1}{2}} [1,\geq -n-\cond(\psi)]. $$
\end{itemize}
\label{CalFourFin}
\end{proposition}
\begin{proof}
	These are elementary computations.
\end{proof}

		\subsubsection{Classical Vectors}
		
	It can be inferred from \cite{Ca73_Br} that the $\gp{K}$-representation
	$$ \Res_{\gp{K}}^{\GL_2(\F)} V_{s,\xi} = \Ind_{\gp{B}(\vo)}^{\gp{K}} (\xi, \omega \xi^{-1}) = \oplus_{n \geq c} V_n $$
	has a decomposition into $\gp{K}$-irreducibles $(\sigma_n,V_n)$ for $c=\cond(\pi_{s,\xi})=\cond(\xi)+\cond(\omega\xi^{-1})$; and that for $N \geq c$
	$$ \oplus_{c \leq n \leq N} V_n = \Ind_{\gp{B}(\vo)}^{\gp{K}} (\xi, \omega \xi^{-1})^{\gp{K}[\vp^N]} $$
	is just the subspace of $V_{s,\xi}$ consisting of $\gp{K}[\vp^N]$-invariant vectors. In each $V_n$ there is a unitary $e_n=e(s,\xi,\omega\xi^{-1};n)$ unique up to a scalar of modulus $1$ which is invariant by $\gp{K}_1[\vp^n]$. Here we have written
	$$ \gp{K}[\vp^n] = \begin{pmatrix} 1+\vp^n & \vp^n \\ \vp^n & 1+\vp^n \end{pmatrix}, \gp{K}_1[\vp^n] = \begin{pmatrix} \vo^{\times} & \vo \\ \vp^n & 1+\vp^n \end{pmatrix}. $$
	Hence $e_n, c \leq n \leq N$ span the subspace of $\gp{K}_1[\vp^N]$-invariant vectors in $V_{s,\xi}$. They are called ``classical vectors'' in \cite{Wu14}.
	
	On the other hand, we have
	$$ \Ind_{\gp{B}(\vo)}^{\gp{K}} (\xi, \omega \xi^{-1}) \simeq (\xi \circ \det) \otimes \Ind_{\gp{B}(\vo)}^{\gp{K}} (1, \omega \xi^{-2}). $$
	Note that $\gp{K}$ transitively acts (at right) on
	$$ \vo \times \vo - \vp \times \vp = \F_1^2 = \{ (x_1,x_2) \in \F^2: \max(\norm[x_1]_{\F},\norm[x_2]_{\F}) = 1 \} \simeq \gp{B}_1(\vo) \backslash \gp{K}, $$
	which induces a $\gp{K}$-isomorphism ($\vo^{\times}$ acting on $\F_1^2$ as scalar multiplication)
	$$ \iota: \Ind_{\vo^{\times}}^{\F_1^2} \omega\xi^{-2} \simeq \Ind_{\gp{B}(\vo)}^{\gp{K}} (1, \omega \xi^{-2}), \Phi \mapsto \left( \kappa \mapsto \int_{\F^{\times}} \Phi((0,t)\kappa) \omega^{-1}\xi^2(t) d^{\times}t = \Cond(\psi)^{-\frac{1}{2}} \Phi((0,1)\kappa) \right). $$
	Note that the $\gp{K}$-invariant measure on $\F_1^2$ is the restriction of the usual Tate's one on $\F^2$, and that
	$$ \Norm[\Phi]^2 = \int_{\F_1^2} \norm[\Phi(x_1,x_2)]^2 dx_1 dx_2 = \zeta_v(2)^{-1} \int_{\gp{K}} \norm[\iota(\Phi)(\kappa)]^2 d\kappa = \zeta_v(2)^{-1} \Norm[\iota(\Phi)]^2 $$
	for our normalization of measures.
\begin{lemma}
	Under the composition of $\iota$ and tensoring by $\xi \circ \det$ map, the space of $\gp{K}_1[\vp^N]$-invariant vectors in $V_{s,\xi}$ is in bijection with $\Phi \in \Ind_{\vo^{\times}}^{\F_1^2} \omega\xi^{-2}$ satisfying
\begin{itemize}
	\item[(1)] $\Phi(u_1x,u_2y) = \xi^{-1}(u_1) \omega\xi^{-1}(u_2) \Phi(x,y)$ for all $u_1,u_2 \in \vo^{\times}, (x,y) \in \F_1^2$;
	\item[(2)] $\Phi(x,y+tx)=\Phi(x,y)$ for all $t \in \vo, (x,y) \in \F_1^2$;
	\item[(3)] $\Phi(x+ty,y)=\Phi(x,y)$ for all $t \in \vp^N, (x,y) \in \F_1^2$.
\end{itemize}
\label{ClassicalPhi}
\end{lemma}
\begin{definition}
	We shall refer to the above subspace of $\Ind_{\vo^{\times}}^{\F_1^2} \omega\xi^{-2}$ as \emph{level $N$} subspace.
\end{definition}
\begin{proof}
	Since $\gp{K}_1[\vp^N]$ is generated by the matrices
	$$ \begin{pmatrix} \vo^{\times} & \\ & \vo^{\times} \end{pmatrix}, \begin{pmatrix} 1 & \vo \\ & 1 \end{pmatrix}, \begin{pmatrix} 1 & \\ \vp^N & 1 \end{pmatrix}, $$
	we conclude by easily translating the action of the above elements on $\Ind_{\vo^{\times}}^{\F_1^2} \omega\xi^{-2}$.
\end{proof}
\begin{lemma}
	For any integer $n \geq c=\cond(\pi(\xi,\omega\xi^{-1}))$ we define $\Phi_n=\Phi_n(\xi,\omega\xi^{-1};\cdot)$ as follows:
\begin{itemize}
	\item[(1)] If $\cond(\xi), \cond(\omega\xi^{-1})>0$, then $c \geq 2$ and
	$$ \Phi_c = \frac{\Cond(\psi)^{\frac{1}{2}} \Cond(\omega\xi^{-1})^{\frac{1}{2}}}{1-q^{-1}} [\xi^{-1}, \cond(\omega\xi^{-1})] \otimes [\omega\xi^{-1},0]; $$
	$$ \Phi_{c+n} = \frac{q^{\frac{n}{2}} \Cond(\psi)^{\frac{1}{2}} \Cond(\omega\xi^{-1})^{\frac{1}{2}}}{1-q^{-1}} [\xi^{-1}, \cond(\omega\xi^{-1})+n] \otimes [\omega\xi^{-1},0], n \geq 1. $$
	\item[(2)] If $\cond(\xi)>0$, $\cond(\omega\xi^{-1})=0$, then $c \geq 1$ and
	$$ \Phi_c = \frac{\Cond(\psi)^{\frac{1}{2}}}{(1-q^{-1})^{\frac{1}{2}}} [\xi^{-1},0] \otimes [1, \geq 0]; \Phi_{c+n} = \frac{q^{\frac{n}{2}} \Cond(\psi)^{\frac{1}{2}}}{(1-q^{-1})(1+q^{-1})^{\frac{1}{2}}} [\xi^{-1},n] \otimes [1,0], n \geq 1. $$
	\item[(3)] If $\cond(\xi)=0$, $\cond(\omega\xi^{-1})>0$, then  $c \geq 1$ and
	$$ \Phi_c = \frac{\Cond(\psi)^{\frac{1}{2}} \Cond(\omega\xi^{-1})^{\frac{1}{2}}}{(1-q^{-1})^{\frac{1}{2}}} [1, \geq \cond(\omega\xi^{-1})] \otimes [\omega\xi^{-1},0]; $$
	$$ \Phi_{c+n} = \frac{q^{\frac{n}{2}} \Cond(\psi)^{\frac{1}{2}} \Cond(\omega\xi^{-1})^{\frac{1}{2}}}{(1-q^{-1})(1+q^{-1})^{\frac{1}{2}}} \left( [1, \geq \cond(\omega\xi^{-1})+n] - q^{-1} [1, \geq \cond(\omega\xi^{-1})+n-1] \right) \otimes [\omega\xi^{-1},0]. $$
	\item[(4)] If $\cond(\xi)=\cond(\omega\xi^{-1})=0$, then $c=0$ and
	$$ \Phi_0 = \frac{\Cond(\psi)^{\frac{1}{2}}}{(1-q^{-2})^{\frac{1}{2}}} 1_{\F_1^2}; \Phi_1 = \Cond(\psi)^{\frac{1}{2}} \sqrt{\frac{q(1+q^{-1})}{1-q^{-1}}} \left( 1_{\vp \times \vo^{\times}} - (q+1)^{-1} 1_{\F_1^2} \right); $$
	$$ \Phi_n = \frac{q^{\frac{n}{2}} \Cond(\psi)^{\frac{1}{2}}}{1-q^{-1}} \left( 1_{\vp^n \times \vo^{\times}} - q^{-1} 1_{\vp^{n-1} \times \vo^{\times}} \right), n \geq 2. $$
\end{itemize}
	Then $\{ \Phi_n: c \leq n \leq N \}$ form an orthonormal basis of the level $N$ subspace.
\end{lemma}
\begin{proof}
	It is easy to see that the action of the upper triangular group $\gp{B}(\vo)$ has orbits in $\F_1^2$ as
	$$ \vo^{\times} \times \vo; (\vp^n-\vp^{n+1}) \times \vo^{\times}, n \geq 1. $$
	Consequently, the orbits of the group $\gp{K}_0[\vp^N]$ for $N \geq 1$ are
	$$ \vo^{\times} \times \vo; (\vp^n-\vp^{n+1}) \times \vo^{\times}, 1 \leq n < N; \vp^N \times \vo^{\times}. $$
	Taking case (4) for example, it follows that $1_{\F_1^2}, 1_{\vp^n \times \vo^{\times}}$ for $1 \leq n \leq N$ form a basis of the level $N$ subspace. We apply the Gram-Schmidt process to conclude.
\end{proof}
\begin{definition}
	We can take $e_n=e(s,\xi,\omega\xi^{-1};n)$ as the flat section defined by
	$$ e(s,\xi,\omega\xi^{-1};n;g) = \zeta_v(2)^{-\frac{1}{2}} \xi(\det g) \norm[\det g]^{\frac{1}{2}+s} \int_{\F^{\times}} \Phi_n(\xi,\omega\xi^{-1};(0,t)g) \omega^{-1}\xi^2(t) \norm[t]^{1+2s} d^{\times}t. $$
\label{NNBasFin}
\end{definition}
		
		\subsubsection{Intertwining Operator}
		
	Since smoothness and $\gp{K}$-finiteness are the same at a finite place, the analytic continuation of the intertwining operator for smooth vectors is the same as the one for $\gp{K}$-finite vectors.
\begin{lemma}
	The effect of $\Intw(s,\xi,\omega\xi^{-1})$ (\ref{LocIntwDef}) on the $\gp{K}$-finite flat sections is determined as follows:
	$$ \IntwR e(s,\xi,\omega\xi^{-1};c+n) = \mu(s,\xi,\omega\xi^{-1};c+n) e(-s,\omega\xi^{-1},\xi;c+n), \forall n \geq 0, \text{ with} $$
\begin{itemize}
	\item[(1)] If $\cond(\xi), \cond(\omega\xi^{-1})>0$, then we have
	$$ \mu(s,\xi,\omega\xi^{-1};c+n) = \frac{L(1+2s,\omega^{-1}\xi^2)}{L(1-2s,\omega\xi^{-2})} \frac{g(\xi^{-1},\psi)\overline{g(\omega^{-1}\xi,\psi)}}{\left( q^n \Cond(\psi) \Cond(\xi) \Cond(\omega\xi^{-1}) \right)^{2s+i\mu(\omega^{-1}\xi^2)}}, \forall n \geq 0. $$
	\item[(2)] If $\cond(\xi)>0$, $\cond(\omega\xi^{-1})=0$, then we have
	$$ \mu(s,\xi,\omega\xi^{-1};c+n) = \frac{g(\xi^{-1},\psi)}{\left( q^n \Cond(\psi) \Cond(\xi) \right)^{2s+i\mu(\omega^{-1}\xi^2)}}, \forall n \geq 0. $$	
	\item[(3)] If $\cond(\xi)=0$, $\cond(\omega\xi^{-1})>0$, then we have
	$$ \mu(s,\xi,\omega\xi^{-1};c+n) = \frac{\overline{g(\omega^{-1}\xi,\psi)}}{\left( q^n \Cond(\psi) \Cond(\xi) \right)^{2s+i\mu(\omega^{-1}\xi^2)}}, \forall n \geq 0. $$	
	\item[(4)] If $\cond(\xi)=\cond(\omega\xi^{-1})=0$, then we have
	$$ \mu(s,\xi,\omega\xi^{-1};0) = \Cond(\psi)^{-(2s+i\mu(\omega^{-1}\xi^2))}; $$
	$$ \mu(s,\xi,\omega\xi^{-1};n) = \frac{1-q^{-(1-2s-i\mu(\omega^{-1}\xi^2))}}{1-q^{-(1+2s+i\mu(\omega^{-1}\xi^2))}} (q^n \Cond(\psi))^{-(2s+i\mu(\omega^{-1}\xi^2))}, \forall n \geq 1. $$
\end{itemize}
	We have for $y \in \ag{R}$, $\norm[\mu(iy,\xi,\omega\xi^{-1};n)]=1$ and
	$$ \extnorm{\mu'(iy,\xi,\omega\xi^{-1};n)} \leq 2 \left( n\log q + \log \Cond(\psi) + \frac{2\log q}{1-q^{-1}} 1_{\cond(\omega^{-1}\xi^2)=0} \right). $$
\label{LocIntwCalFin}
\end{lemma}
\begin{proof}
	This is a combination of definitions (\ref{LocIntwDef}), Definition \ref{NNBasFin} and the computation Proposition \ref{CalFourFin}. We take the case (4) for example. Writing $\Psi_n=[1,\geq n] \otimes [1,\geq 0]$, we can define
	$$ f_n(s,\xi,\omega\xi^{-1};\kappa) = \int_{\F^{\times}} \Psi_n((0,t)\kappa) \norm[t]_{\F}^{1+2s} \omega^{-1}\xi^2(t) d^{\times}t. $$
	On the one hand, letting $\mu=\mu(\omega^{-1}\xi^2)$ we have by definition
\begin{align*}
	\Intw f_n(s,\xi,\omega\xi^{-1};\kappa) &= \int_{\F^{\times}} \widehat{\Psi_n}((0,t)\kappa) \norm[t]_{\F}^{1-2s} \omega\xi^{-2}(t) d^{\times}t \\
	&= q^{-n} \Cond(\psi)^{-1} \int_{\F^{\times}} \Psi_n((0,t \varpi_{\F}^{n+\cond(\psi)})\kappa) \norm[t]_{\F}^{1-2s} \omega\xi^{-2}(t) d^{\times}t \\
	&= (q^n \Cond(\psi))^{-(2s+i\mu)} f_n(-s,\omega\xi^{-1},\xi;\kappa), \forall \kappa \in \gp{K}.
\end{align*}
	On the other hand, we have
	$$ \Psi_n = \sum_{k=0}^{n-1} \varpi_{\F}^{-k}.1_{\vp^{n-k} \times \vo^{\times}} + \sum_{k=n}^{\infty} \varpi_{\F}^{-k}.1_{\F_1^2}, $$
	which implies, writing $f_n=f_n(s,\xi,\omega\xi^{-1})$,
	$$ f_n = \sum_{k=0}^{n-1} q^{-k(1-2s-i\mu)} \iota(1_{\vp^{n-k} \times \vo^{\times}}) + \frac{q^{-n(1-2s-i\mu)}}{1-q^{-(1-2s-i\mu)}} \iota(1_{\F_1^2}). $$
	Thus we get
	$$ \iota(1_{\F_1^2}) = (1-q^{-(1-2s-i\mu)}) f_0, \iota(1_{\vp^n \times \vo^{\times}}) = f_n - q^{-(1-2s-i\mu)} f_{n-1}, \forall n \geq 1. $$
	Consequently, we have the equations
	$$ \iota(1_{\vp \times \vo^{\times}}) - (q+1)^{-1} \iota(1_{\F_1^2}) = f_1 - (q+1)^{-1} (1+q^{-(2s+i\mu)}) f_0, $$
	$$ \iota(1_{\vp^n \times \vo^{\times}}) - q^{-1} \iota(1_{\vp^{n-1} \times \vo^{\times}}) = f_n - q^{-1}(1+q^{-(2s+i\mu)}) f_{n-1} + q^{-2(1+s+i\mu)} f_{n-2}, \forall n \geq 2. $$
	We finally conclude by the relation of $e_n$ and $\Phi_n$, and the computation
	$$ \Intw e(s,\xi,\omega\xi^{-1};0) = \frac{1-q^{-(1+2s+i\mu)}}{1-q^{-(1-2s-i\mu)}} \Cond(\psi)^{-(2s+i\mu)} e(-s,\omega\xi^{-1},\xi;0); $$
\begin{align*}
	\Intw e(s,\xi,\omega\xi^{-1};1) &= \Cond(\psi)^{\frac{1}{2}} \sqrt{\frac{q(1+q^{-1})}{1-q^{-1}}} \left\{ (q \Cond(\psi))^{-(2s+i\mu)} f_1(-s,\omega\xi^{-1},\xi) \right. \\
	&\quad \left. - (q+1)^{-1} (1+q^{-(2s+i\mu)}) \Cond(\psi)^{-(2s+i\mu)} f_0(-s,\omega\xi^{-1},\xi) \right\} \\
	&= (q \Cond(\psi))^{-(2s+i\mu)} \Cond(\psi)^{\frac{1}{2}} \sqrt{\frac{q(1+q^{-1})}{1-q^{-1}}} \left\{ f_1(-s,\omega\xi^{-1},\xi) \right. \\
	&\quad \left. - (q+1)^{-1} (1+q^{2s+i\mu}) f_1(-s,\omega\xi^{-1},\xi) \right\} \\
	&= (q \Cond(\psi))^{-(2s+i\mu)} e(-s,\omega\xi^{-1},\xi;1);
\end{align*}
\begin{align*}
	\Intw e(s,\xi,\omega\xi^{-1};n) &= \frac{q^{\frac{n}{2}} \Cond(\psi)^{\frac{1}{2}}}{1-q^{-1}} \left\{ (q^n \Cond(\psi))^{-(2s+i\mu)} f_n(-s,\omega\xi^{-1},\xi) \right. \\
	&\quad - q^{-1}(1+q^{-(2s+i\mu)}) (q^{n-1} \Cond(\psi))^{-(2s+i\mu)} f_{n-1}(-s,\omega\xi^{-1},\xi) \\
	&\quad \left. + q^{-2}q^{-(2s+i\mu)} (q^{n-2} \Cond(\psi))^{-(2s+i\mu)} f_{n-2}(-s,\omega\xi^{-1},\xi) \right\} \\
	&= (q^n \Cond(\psi))^{-(2s+i\mu)} e(-s,\omega\xi^{-1},\xi;n).
\end{align*}
	The asserted bound for $\mu'(iy,\xi,\omega\xi^{-1};n)$ is then obvious from the formulae.
\end{proof}

	\subsection{Global Estimation}
	
	We are ready to study the size of a $\gp{K}$-isotypic Eisenstein series. For simplicity of notations we write
	$$ \Eis(s,\xi,\omega\xi^{-1};\vec{n}) = \Eis(s,\xi,\omega\xi^{-1};e_{\vec{n}}) $$
where
\begin{itemize}
	\item $\vec{n}=(n_v)_v \in \oplus_{v \in \Place(\F)} \ag{Z}$ with $n_v$ parametrizing the $\gp{K}_v$-type as before;
	\item $e_{\vec{n}} = e_{\vec{n}}(\xi,\omega\xi^{-1}) = \otimes_v' e_{n_v}$ with $e_{n_v}$ being a unitary vector in the $\gp{K}_v$-isotypic part of $V_{0,\xi_v,\omega_v\xi_v^{-1}}$ with parameter $n_v$.
\end{itemize}
	Note that $\vec{n}$ does not determine $e_{\vec{n}}$. In fact, the dimension $d(\xi,\omega\xi^{-1};\vec{n})$ of $\gp{K}$-type $\vec{n}$ vectors in $V_{s,\xi,\omega\xi^{-1}}$ satisfies
\begin{itemize}
	\item $d(\xi,\omega\xi^{-1};\vec{n}) = \Pi_v d(\xi_v,\omega_v\xi_v^{-1};n_v)$, where $d(\xi_v,\omega_v\xi_v^{-1};n_v)$ is the dimension of $\gp{K}_v$-type $n_v$ vectors;
	\item $d(\xi_v,\omega_v\xi_v^{-1};n_v) \leq n_v+1$ for $v \mid \infty$ complex and $d(\xi_v,\omega_v\xi_v^{-1};n_v) \leq 1$ for $v \mid \infty$ real;
	\item $d(\xi_v,\omega_v\xi_v^{-1};n_v) = (q_v^{n_v}-q_v^{n_v-2}1_{n_v \geq 2})1_{n_v \geq \cond(\xi_v)+\cond(\omega_v\xi_v^{-1})}$ for $v < \infty$.
\end{itemize}
	Recall the height function $\Ht$ defined by
	$$ \Ht: \GL_2(\ag{A}) \to \ag{R}_+, \begin{pmatrix} t_1 & x \\ & t_2 \end{pmatrix} \kappa \mapsto \extnorm{\frac{t_1}{t_2}}_{\ag{A}}, \forall t_1,t_2 \in \ag{A}^{\times}, x \in \ag{A}, \kappa \in \gp{K}. $$
\begin{lemma}
	If $\gamma \in \GL_2(\F) - \gp{B}(\F)$, then we have $\Ht(\gamma g) \leq \Ht(g)^{-1}$.
\label{ShiftHtBd}
\end{lemma}
\begin{proof}
	As we can take $\gamma=wn(\alpha)$ for some $\alpha \in \F$, it suffices to prove $\Ht(wn(x)) \leq 1$ for all $x \in \ag{A}$. In fact locally at each $v$, $\Ht_v(wn(x_v)) \leq 1$. We leave this simple verification to the reader.
\end{proof}
\begin{lemma}
	Let $K$ be a compact subset of $\GL_2(\F) \gp{Z}(\ag{A}) \backslash \GL_2(\ag{A})$. Then there is a constant $C=C(K,\F,(n_v)_{v < \infty},\omega^{-1}\xi^2)$ such that for $y \in \ag{R}$
	$$ \int_{K} \norm[\Eis(iy,\xi,\omega\xi^{-1};\vec{n})(g)]^2 dg \ll C+ \log \Cond(\omega^{-1}\xi^2 \norm_{\ag{A}}^{2iy}) + \sum_{\substack{v \mid \infty \\ n_v \neq 0}} \log (\norm[n_v]+1). $$
	$C$ depends on $\omega^{-1}\xi^2$ only if the later is quadratic with a non trivial Siegel zero.
\end{lemma}
\begin{proof}
	By the reduction theory, there is $0<c_{\F}<1$ depending only on $\F$ such that the Siegel domain $S_{c_{\F}}$ contains a fundamental domain of $\GL_2(\F) \gp{Z}(\ag{A}) \backslash \GL_2(\ag{A})$. Fixing $K$, we can find $C > c_{\F}^{-1}$ and assume $K \subset S_{c_{\F}} - S_C$. Consequently, we get by Lemma \ref{ShiftHtBd}
	$$ \Eis(iy,\xi,\omega\xi^{-1};\vec{n})(g) - \Lambda^C \Eis(iy,\xi,\omega\xi^{-1};\vec{n})(g) = \left\{ \begin{matrix} 0 & \text{if } g \in S_{c_{\F}} - S_C \\ \Eis(iy,\xi,\omega\xi^{-1};\vec{n})_{\gp{N}}(g) & \text{if } g \in S_C. \end{matrix} \right. $$
	In particular, we can apply (\ref{MSRel}) and get
\begin{align*}
	\int_{K} \norm[\Eis(iy,\xi,\omega\xi^{-1};\vec{n})(g)]^2 dg &= \int_{K} \norm[\Lambda^C\Eis(iy,\xi,\omega\xi^{-1};\vec{n})(g)]^2 dg \leq \extNorm{\Lambda^C\Eis(iy,\xi,\omega\xi^{-1};\vec{n})}^2 \\
	&= 2 \log C - \Pairing{\Intw(-iy,\omega\xi^{-1},\xi) \Intw'(iy,\xi,\omega\xi^{-1})e_{\vec{n}}}{e_{\vec{n}}}_0 \\
	&\quad + \frac{1_{\omega=\xi^2}}{y} \Im \left( C^{2iy} \Pairing{e_{\vec{n}}}{\Intw(iy,\xi,\xi)e_{\vec{n}}}_0 \right),
\end{align*}
	where $\Intw'$ is the partial derivative of $\Intw$ with respect to $s$, and the pairing $\Pairing{\cdot}{\cdot}_0$ is taken by identifying everything in the Hilbert space $\Res_{\gp{K}}^{\GL_2(\ag{A})} V_{0,\xi,\omega\xi^{-1}}$. Note that we have used Remark \ref{CharNorm}, i.e. $\omega^{-1}\xi^2(\ag{A}^{(1)})=1 \Rightarrow \mu(\omega^{-1}\xi^2)=0$ (hence $\omega=\xi^2$). From the calculations in Corollary \ref{LocIntwCalCpx}, \ref{LocIntwCalReal} and Lemma \ref{LocIntwCalFin}, if we write
	$$ \mu(s,\xi,\omega\xi^{-1};\vec{n}) = \mu_{\F}(s,\omega^{-1}\xi^2) \prod_v \mu_v(s,\xi_v,\omega_v\xi^{-1};n_v) \text{ with} $$
	$$ \mu_{\F}(s,\omega^{-1}\xi^2) = \frac{\Lambda(1-2s, \omega\xi^{-2})}{\Lambda(1+2s, \omega^{-1}\xi^2)} $$
	and $\mu_v(s,\xi_v,\omega_v\xi^{-1};n_v)$ defined locally in the above mentioned lemmas, then we find
	$$ \Pairing{\Intw(-iy,\omega\xi^{-1},\xi) \Intw'(iy,\xi,\omega\xi^{-1})e_{\vec{n}}}{e_{\vec{n}}}_0 = \frac{\mu'(iy,\xi,\omega\xi^{-1};\vec{n})}{\mu(iy,\xi,\omega\xi^{-1};\vec{n})}, $$
	$$ \Pairing{e_{\vec{n}}}{\Intw(iy,\xi,\xi)e_{\vec{n}}}_0 = \overline{\mu(iy,\xi,\xi;\vec{n})}. $$
	It is a classical and subtle result of the analytic number theory that
\begin{equation}
	\extnorm{\frac{\Lambda'(1+2iy,\chi)}{\Lambda(1+2iy,\chi)}} \ll \log \Cond(\chi \norm_{\ag{A}}^{2iy}) + [\F:\ag{Q}] + \frac{1}{1-\beta_{\chi}},
\label{LogDerLBd}
\end{equation}
	where the last term is present only if $\chi\norm_{\ag{A}}^{i\alpha_{\chi}}, \alpha_{\chi} \in \ag{R}$ is quadratic admitting a Siegel zero $\beta_{\chi}$.
	It follows that
	$$ \extnorm{\frac{\mu_{\F}'(iy,\omega^{-1}\xi^2)}{\mu_{\F}(iy,\omega^{-1}\xi^2)}} \ll \log \Cond(\omega^{-1}\xi^2 \norm_{\ag{A}}^{2iy}) + [\F:\ag{Q}] + \frac{1}{1-\beta_{\omega^{-1}\xi^2}}. $$
	Together with the local bounds obtained in Corollary \ref{LocIntwCalCpx}, \ref{LocIntwCalReal} and Lemma \ref{LocIntwCalFin}, we deduce that
\begin{align*}
	\extnorm{\frac{\mu'(iy,\xi,\omega\xi^{-1};\vec{n})}{\mu(iy,\xi,\omega\xi^{-1};\vec{n})}} &\ll \sum_{\substack{v \mid \infty \\ n_v \neq 0}} \log (\norm[n_v]+1) + \sum_{\substack{v < \infty \\ n_v \neq 0}} (n_v+2) \log q_v + \frac{1}{1-\beta_{\omega^{-1}\xi^2}} \\
	&\quad + [\F:\ag{Q}] + \log \Cond(\omega^{-1}\xi^2 \norm_{\ag{A}}^{2iy}).
\end{align*}
	Since we have $\mu(0,\xi,\xi;\vec{n})=-1$ and
	$$ \extnorm{\frac{C^{2iy}-C^{-2iy}}{2y}} = \log C \cdot \extnorm{\int_{-1}^1 e^{2ity \log C} dt} \leq 2 \log C, $$
	$$ \extnorm{\frac{\mu(iy,\xi,\xi;\vec{n})+1}{y}} = \extnorm{\int_0^1 \mu'(ity,\xi,\xi;\vec{n}) dt} \leq \max_{0 \leq t \leq 1} \norm[\mu'(ity,\xi,\xi;\vec{n})], $$
	we deduce that
\begin{align*}
	\extnorm{\frac{\Im ( C^{2iy} \overline{\mu(iy,\xi,\xi;\vec{n})} )}{y}} &\ll \min \left( 2 \log C, \frac{1}{\norm[y]} \right) \\
	&+ \min \left( [\F:\ag{Q}] + \log \Cond(\norm_{\ag{A}}^{2iy}) + \sum_{\substack{v \mid \infty \\ n_v \neq 0}} \log (\norm[n_v]+1) + \sum_{\substack{v < \infty \\ n_v \neq 0}} (n_v+2) \log q_v, \frac{1}{\norm[y]} \right).
\end{align*}
	Combining the two bounds together, we conclude.
\end{proof}
\begin{remark}
	The concerned bound (\ref{LogDerLBd}) follows from, for example, either the argument \cite[Proposition 5.7 (2)]{IK04} or \cite[Theorem 16]{Te95}, both being based on the zero-free region \cite[Theorem 5.35]{IK04}. The tricky Siegel zero is the source of the famous ineffectiveness of Siegel's lower bound.
\end{remark}

	Recall that at each $v \mid \infty$, we have a Laplacian operator
	$$ \Delta_v = - \Casimir_{\gp{Z}(\F_v) \backslash \GL_2(\F_v)} - 2 \Casimir_{\gp{K}_v}, $$
	where $\Casimir_{\gp{G}}$ denotes the Casimir element of the Lie group $\gp{G}$.
\begin{itemize}
	\item For $\F_v=\ag{C}$, $\Delta_v$ acts on the $\gp{K}_v$-isotypic type-$n_v$ subspace of $V_{iy,\xi_v,\omega_v\xi_v^{-1}}$ as multiplication by
	$$ \frac{1+(2y+\mu(\omega_v^{-1}\xi_v^2))^2}{4}+\frac{2n_v(n_v+2)-n_0^2}{4}, $$
	where $n_0=n_0(\omega_v^{-1}\xi_v^2) \in \ag{Z}$ is such that $\omega_v^{-1}\xi_v^2(e^{i\alpha})=e^{in_0\alpha}$.
	\item For $\F_v=\ag{R}$, $\Delta_v$ acts on the $\gp{K}_v$-isotypic type-$n_v$ subspace of $V_{iy,\xi_v,\omega_v\xi_v^{-1}}$ as multiplication by
	$$ \frac{1+(2y+\mu(\omega_v^{-1}\xi_v^2))^2}{4}+\frac{n_v^2}{2}. $$
\end{itemize}
	We get an elliptic operator on $\Pi_{v \mid \infty} \gp{Z}(\F_v) \backslash \GL_2(\F_v)$
	$$ \Delta_{\infty} = \sum_{v \mid \infty} \Delta_v. $$
	By Sobolev's inequalities, there is $N \in \ag{N}$ such that for any compact subset $K \subset \GL_2(\F)\gp{Z}(\ag{A}) \backslash \GL_2(\ag{A})$ we have for smooth function $\varphi$
	$$ \sup_{g \in K} \norm[\varphi(g)] \ll_K \left( \int_K \extnorm{(1+\Delta_{\infty})^N.\varphi(g)}^2 dg \right)^{\frac{1}{2}}. $$
\begin{corollary}
	Let $\lambda(iy,\xi,\omega\xi^{-1};\vec{n}_{\infty})$ be the eigenvalue of $\Delta_{\infty}$ on $\Eis(iy,\xi,\omega\xi^{-1};\vec{n})$. Under the condition of the lemma, we have
	$$ \sup_{g \in K} \norm[\Eis(iy,\xi,\omega\xi^{-1};\vec{n})(g)] \ll_{K,\F,(n_v)_{v<\infty},\omega^{-1}\xi^2} (1+\lambda(iy,\xi,\omega\xi^{-1};\vec{n}_{\infty}))^N \log (1+\lambda(iy,\xi,\omega\xi^{-1};\vec{n}_{\infty})). $$
\label{SobBdContKiso}
\end{corollary}
\begin{proof}
	We only need to notice that the right hand side of the inequality in the lemma is $\ll \log (1+\lambda(iy,\xi,\omega\xi^{-1};\vec{n}_{\infty}))$.
\end{proof}
\begin{remark}
	It is easier to get the similar bounds
	$$ \sup_{g \in K} \norm[\Eis(\sigma+iy,\xi,\omega\xi^{-1};\vec{n})(g)] \ll_{K,\F,(n_v)_{v<\infty},\omega^{-1}\xi^2} (\norm[y]+\Norm[\vec{n}])^N $$
	for $\norm[y] \gg 1$ and uniform for $\sigma$ lying in any compact sub-interval of $\ag{R}$.
\end{remark}

	We are now ready to prove the main results.
\begin{proof}{(Proposition \ref{FourInvIncompTheta})}
	We introduce the Fourier coefficients $\hat{f}(s,\xi,\omega\xi^{-1};\vec{n})$ which are entire functions in $s$ such that
	$$ \hat{f}(s,\xi,\omega\xi^{-1})(\kappa) = \sum_{\vec{n}} \hat{f}(s,\xi,\omega\xi^{-1};\vec{n}) e_{\vec{n}}(\xi,\omega\xi^{-1};\kappa), \forall \kappa \in \gp{K} $$
with bound for any $A > 0$ and $\Re s$ lying in any fix compact set
	$$ \extnorm{\hat{f}(s,\xi,\omega\xi^{-1};\vec{n})} \ll_{f,A} (\norm[\Im s] + \Norm[\vec{n}])^{-A}. $$
We then have uniformly for $\Re s \geq c \gg 1$ and with normal convergence in $g$
	$$ \Eis(\hat{f}(s,\xi,\omega\xi^{-1};\vec{n}))(g) = \sum_{\vec{n}} \hat{f}(s,\xi,\omega\xi^{-1};\vec{n}) \Eis(s,\xi,\omega\xi^{-1};\vec{n})(g). $$
It then holds for $s$ lying in the complement of the possible poles of $\Eis(s,\xi,\omega\xi^{-1};\vec{n})$ by the uniqueness of analytic continuation and the normal convergence of the right hand side. Hence we get
	$$ P(f)(g) = \sum_{\xi} \sum_{\vec{n}} \int_{\Re s = c} \hat{f}(s,\xi,\omega\xi^{-1};\vec{n}) \Eis(s,\xi,\omega\xi^{-1};\vec{n})(g) \frac{ds}{2\pi i}, $$
and we can shift the integral to $\Re s=0$ and get the desired strong Fourier inversion formula.
\end{proof}
\begin{lemma}
	(S3) is verified for $\Cont_c^{\infty}(\GL_2,\omega)$. Precisely, for $h \in \Cont_c^{\infty}(\GL_2,\omega)$ define
	$$ \ProjG_{iy,\xi,\omega\xi^{-1}}(h) = \sum_{\vec{n}} \Pairing{h}{\Eis(iy,\xi,\omega\xi^{-1};\vec{n})} e(iy,\xi,\omega\xi^{-1}; \vec{n}) \in V_{iy,\xi,\omega\xi^{-1}}^{\infty}, $$
where $e(s,\xi,\omega\xi^{-1}; \vec{n})$ is the flat section induced by $e_{\vec{n}}=e_{\vec{n}}(\xi,\omega\xi^{-1})$. Then we have for almost every $y$
	$$ \ProjG_{iy,\xi,\omega\xi^{-1}}(h) = \ProjF_{iy,\xi,\omega\xi^{-1}}(h) $$
as vectors in $V_{iy,\xi,\omega\xi^{-1}}$.
\label{FFforCcInf}
\end{lemma}
\begin{proof}
	By integration by parts we get for any $A \in \ag{N}$
	$$ \Pairing{h}{\Eis(iy,\xi,\omega\xi^{-1};\vec{n})} = \lambda(iy,\xi,\omega\xi^{-1};\vec{n}_{\infty})^{-A} \Pairing{\Delta_{\infty}^A h}{\Eis(iy,\xi,\omega\xi^{-1};\vec{n})}, $$
which is rapidly decreasing in $(y,\vec{n})$ using Corollary \ref{SobBdContKiso}. Hence $\ProjG_{iy,\xi,\omega\xi^{-1}}(h)$ is converging, well-defined and rapidly decreasing in $y$ as a measurable section. By Proposition \ref{FourInvIncompTheta} and Fubini, we can insert the Fourier expansion of an incomplete theta series and interchange the summation to get
	$$ \Pairing{h}{\Pcare(f)} = \sum_{\xi} \int_0^{\infty} \Pairing{\ProjG_{iy,\xi,\omega\xi^{-1}}(h)}{\ProjF_{iy,\xi,\omega\xi^{-1}}(\Pcare(f))}_{iy,\xi,\omega\xi^{-1}} \frac{dy}{2\pi} + \cdots. $$
	On the other hand, the Plancherel formula gives
	$$ \Pairing{h}{\Pcare(f)} = \sum_{\xi} \int_0^{\infty} \Pairing{\ProjF_{iy,\xi,\omega\xi^{-1}}(h)}{\ProjF_{iy,\xi,\omega\xi^{-1}}(\Pcare(f))}_{iy,\xi,\omega\xi^{-1}} \frac{dy}{2\pi} + \cdots. $$
	The rapid decay of the section $\ProjG_{iy,\xi,\omega\xi^{-1}}(h)$ implies its square-integrability. Hence $\ProjG_{iy,\xi,\omega\xi^{-1}}(h) - \ProjF_{iy,\xi,\omega\xi^{-1}}(h)$ is a square-integrable section orthogonal to the sections $\ProjF_{iy,\xi,\omega\xi^{-1}}(\Pcare(f))$ for all $\Pcare(f)$. As $\Pcare(f)$ is dense in the continuous part, we must have $\ProjG_{iy,\xi,\omega\xi^{-1}}(h) - \ProjF_{iy,\xi,\omega\xi^{-1}}(h)=0$ as square-integrable sections, which implies the equality almost everywhere in $y$.
\end{proof}
\begin{lemma}
	Let $\varphi \in V_{\rpR_{\omega}}^{\infty}$ be represented by a smooth function. Then the Fourier expansion of $\varphi$
	$$ \sum_{\pi \text{ cusp}} \sum_{\vec{n}} \ProjP_{\pi}(\varphi)[\vec{n}](g) + \sum_{\chi^2=\omega} \ProjP_{\chi \circ \det}(\varphi)(g) + \sum_{\xi} \int_0^{\infty} \sum_{\vec{n}} \ProjP_{iy,\xi,\omega\xi^{-1}}(\varphi)[\vec{n}](g) \frac{dy}{2\pi} $$
converges normally to a continuous (smooth) function. Here ``$[\vec{n}]$'' is a natural parametrization of $\gp{K}$-isotypic types defined in Definition \ref{KProjParam}.
\label{NConvFES}
\end{lemma}
\begin{proof}
	Similar to the continuous part, let $\lambda(\pi;\vec{n}_{\infty})$ be the eigenvalue of $\Delta_{\infty}$ on $\pi[\vec{n}]$. By Corollary \ref{SobBdContKiso} and the usual Sobolev's inequalities, we have for any given compact subset $K$
\begin{align*}
	&\quad \sum_{\pi \text{ cusp}} \sum_{\vec{n}} \extnorm{\ProjP_{\pi}(\varphi)[\vec{n}](g)} + \sum_{\chi^2=\omega} \extnorm{\ProjP_{\chi \circ \det}(\varphi)(g)} + \sum_{\xi} \int_0^{\infty} \sum_{\vec{n}} \extnorm{\ProjP_{iy,\xi,\omega\xi^{-1}}(\varphi)[\vec{n}](g)} \frac{dy}{2\pi} \\
	&\ll_{K,(n_v)_{v<\infty}} \sum_{\pi \text{ cusp}} \sum_{\vec{n}} (1+\lambda(\pi;\vec{n}_{\infty}))^N \extNorm{\ProjF_{\pi}(\varphi)[\vec{n}]} + \sum_{\chi^2=\omega} \extNorm{\ProjF_{\chi \circ \det}(\varphi)} + \\
	&\quad \sum_{\xi} \int_0^{\infty} \sum_{\vec{n}} (1+\lambda(iy,\xi,\omega\xi^{-1};\vec{n}_{\infty}))^{N+1} \extNorm{\ProjF_{iy,\xi,\omega\xi^{-1}}(\varphi)[\vec{n}]} \frac{dy}{2\pi} \\
	&\leq \Norm[(1+\Delta_{\infty})^{N+A}\varphi] \cdot \left( \sum_{\pi \text{ cusp}} \sum_{\vec{n}} (1+\lambda(\pi;\vec{n}_{\infty}))^{-2A} + \sum_{\xi} \int_0^{\infty} \sum_{\vec{n}} (1+\lambda(iy,\xi,\omega\xi^{-1};\vec{n}_{\infty}))^{2-2A} \frac{dy}{2\pi} \right)^{\frac{1}{2}} \\
	&\quad + \sum_{\chi^2=\omega} \extNorm{\ProjF_{\chi \circ \det}(\varphi)}.
\end{align*}
	For large enough $A$, the convergence of the integral concerning $\lambda(iy,\xi,\omega\xi^{-1};\vec{n}_{\infty})$ follows from the explicit calculation of its local components given above; while the convergence of the sum concerning $\lambda(\pi;\vec{n}_{\infty})$ is a weak Weyl law \cite[Theorem 2.23]{Wu14}, the main theme of \cite{Pa12}.
\end{proof}
\begin{proof}{(Theorem \ref{MainThm})}
	We are now in a situation completely analogous to the case of the classical Fourier analysis. Write $\tilde{\varphi}$ for the Fourier expansion of $\varphi$. By Lemma \ref{NConvFES} and Fubini we interchange the summations to get for any $h \in \Cont_c^{\infty}(\GL_2,\omega)$
\begin{align*}
	\Pairing{\tilde{\varphi}}{h} &= \sum_{\pi \text{ cusp}} \Pairing{\ProjF_{\pi}(\varphi)}{\ProjF_{\pi}(h)} + \sum_{\chi^2=\omega} \Pairing{\ProjF_{\chi \circ \det}(\varphi)}{\ProjF_{\chi \circ \det}(h)} \\
	&\quad + \sum_{\xi} \int_0^{\infty} \Pairing{\ProjF_{iy,\xi,\omega\xi^{-1}}(\varphi)}{\ProjG_{iy,\xi,\omega\xi^{-1}}(h)}_{iy,\xi,\omega\xi^{-1}} \frac{dy}{2\pi}.
\end{align*}
But we can replace $\ProjG_{iy,\xi,\omega\xi^{-1}}(h)$ with $\ProjF_{iy,\xi,\omega\xi^{-1}}(h)$ by Lemma \ref{FFforCcInf}. Then the Plancherel formula gives
	$$ \Pairing{\tilde{\varphi}}{h} = \Pairing{\varphi}{h}, \forall h \in \Cont_c^{\infty}(\GL_2,\omega), $$
which means $\tilde{\varphi}$ and $\varphi$ are equal in the sense of distributions. Both being continuous functions, they must be equal pointwise.
\end{proof}

\section{An Incomplete ``Proof''}

	There is an incomplete argument, which is our understanding of the ambiguous ``proof'' of \cite[Theorem 1.1]{CP90} in the direction of Fourier inversion. It is guided by the following beautiful functional analytic viewpoint on Fourier inversion. First let's go back to the formalism of spectral decomposition. We can consider formulae more general than the Plancherel formula by writing
\begin{equation}
	\Pairing{\ell}{v} = \int_{\widehat{\gp{G}}} \Pairing{\ProjF_{\pi}(\ell)}{\ProjF_{\pi}(v)} d\mu(\pi),
\label{SpecDell}
\end{equation}
\noindent where $v \in W \subset V_{\rpR}$ is a subspace with its own topology and $\ell \in W^*$ in the topological dual of $W$. Of course, we must give $\ProjF_{\pi}(\ell)$ suitable definition. On the other hand, in the setting of Fourier inversion with $X=\Gamma \backslash \gp{G}$ a homogeneous space, we ask formulae like
\begin{equation}
	\Pairing{\delta_{\Gamma e}}{v} = \int_{\widehat{\gp{G}}} \Pairing{\delta_{\Gamma e}}{\FuncRA_{\pi} \circ \ProjF_{\pi}(v)} d\mu(\pi).
\label{FourInvDelta}
\end{equation}
\noindent Specializing to the case $\gp{G}$ is a Lie group and $\Gamma$ is a lattice, we can use the Lie algebra to define the Sobolev spaces $W^{p,q}(X) \subset \intL^q(X,dx)$. In particular, the space of smooth vectors $V_{\rpR}^{\infty}$ is just $W^{\infty,2}(X)$. Linearizing around $\Gamma e$, we find that $\delta_{\Gamma e} \in W^{-\infty,2}(X)$. A general theorem in the theory of Sobolev spaces states that any functional in $W^{-\infty,2}(X)$ is just a finite combination of weak derivatives of elements in $\intL^2(X,dx)$. So (\ref{SpecDell}) makes sense for $\ell=\delta_{\Gamma e}$. It remains to check for a.e. $\pi$
	$$ \Pairing{\ProjF_{\pi}(\delta_{\Gamma e})}{\ProjF_{\pi}(v)} = \Pairing{\delta_{\Gamma e}}{\FuncRA_{\pi} \circ \ProjF_{\pi}(v)}, v \in W^{\infty,2}(X). $$
Writing explicitly $\delta_{\Gamma e}$ as weak derivatives (of a smooth transitioned Heaviside function + a $\Cont_c^{\infty}$ function), we are lead to check
	$$ \Pairing{\ProjF_{\pi}(v)}{\ProjF_{\pi}(v_0)} = \Pairing{\FuncRA_{\pi} \circ \ProjF_{\pi}(v)}{v_0}, v \in W^{\infty,2}(X) $$
for some $v_0 \in \Cont_c^{\infty}(X-\{ \Gamma e \})$ with singularity at $\Gamma e$. Since $\ProjF_{\pi}(W^{\infty,2}(X)) \subset V_{\pi}^{\infty}$, we only need to check
	$$ \Pairing{v}{\ProjF_{\pi}(v_0)} = \Pairing{\FuncRA_{\pi}(v)}{v_0}, v \in V_{\pi}^{\infty}. $$
We know the validity of the above equation if we restrict $v_0 \in V$ (c.f. Remark \ref{IncompThetaAdj}) or even for $v_0 \in \Cont_c^{\infty}(\GL_2,\omega)$ (Lemma \ref{FFforCcInf}), but we do not know how to extend it to the $v_0$ with singularity which interests us.


	There is another way to understand \cite[Theorem 1.1]{CP90}, which applies to a special class of functionals in $W_0^* \subset W^{-\infty,2}(X)$ and vectors in $W^{\infty,2}(X)$. Assume
\begin{itemize}
	\item we have a closed subgroup $\gp{H} < \gp{G}$ and a (quasi-)character $\chi: \gp{H} \to \ag{C}^{\times}$ such that $\ell$ is $(\gp{H},\chi)$-covariant, i.e.,
	$$ \ell(\rpR(h).v) = \chi(h) \ell(v), \forall h \in \gp{H}, v \in W^{\infty,2}(X). $$
	\item the space of $(\gp{H},\chi)$-covariant functionals has multiplicity one property, i.e., for any unitary irreducible representation $\pi$ of $\gp{G}$, the space of continuous functionals $\ell$ on $V_{\pi}^{\infty}$ satisfying
	$$ \ell(\pi(h).v) = \chi(h) \ell(v), \forall h \in \gp{H}, v \in V_{\pi}^{\infty} $$
is of dimension at most $1$, and we fix a non trivial one $\ell_{\pi}$ for each $\pi$.
\end{itemize}
Then from $\ell(\rpR(h).v) = \chi(h) \ell(v)$ and the interpretation of $\ell$ as weak derivatives we could deduce
	$$ \Pairing{\ProjF_{\pi}(\ell)}{\pi(h).v} = \chi(h) \Pairing{\ProjF_{\pi}(\ell)}{v}, \forall v \in V_{\pi}^{\infty}. $$
Hence $\ProjF_{\pi}(\ell) = c(\pi) \ell_{\pi}$ for some $c(\pi) \in \ag{C}$ by multiplicity one. Then we apply $\ell$ to $v \in V$ and use Proposition \ref{FourInvIncompTheta} to determine $c(\pi)$. This seems to be the correct way to understand \cite[Theorem 1.1]{CP90}. However, in this way the concerned theorem does not imply Fourier inversion.

\section*{Acknowledgement}

	The author would like to thank Daniel Wuersch and Manuel Luethi, without the questions from whom he would never have decided to write a note on this problem. The author would also like to thank Pengyu Le, without the several detailled discussions with whom he would never have decided to write the note in the form of an article. The author is grateful for the suggestive discussions with Prof. Paul Nelson at ETHZ, Prof. Philippe Michel at EPFL, Prof. Zhiying Wen \& Prof. Jiayan Yao \& Prof. Yanhui QU at Tsinghua University, and Dr. Yanqi Qiu in CNRS at IMT during the preparation of this paper. The author is also grateful to the math department at Xi'an Jiao Tong University for giving him the chance to talk about the results in this paper in public. The preparation of the paper begins at the end of the author's stay in FIM at ETHZ and ends at the beginning of the author's stay in YMSC at Tsinghua University. The author would like to thank both institutes for their hospitality.

\bibliographystyle{acm}

\bibliography{mathbib}

\address{\quad \\ Han WU \\ 302, Jin Chun Yuan West Building \\ YMSC, Tsinghua University \\ 10084, Beijing \\ China \\ wuhan1121@yahoo.com}

\end{document}